\documentclass[preprint,10pt]{elsarticle}

\usepackage{amssymb,amsmath,amsfonts,latexsym}
\usepackage{amsthm}
\usepackage{mathtools}
\usepackage{multirow,array}
\usepackage{graphicx}
\usepackage{subfigure}
\usepackage{epstopdf}
\usepackage{float} 
\usepackage{color, xcolor}
\usepackage{hyperref}

\renewcommand{\thesection}{\arabic{section}}
\newtheorem{thm}{\thesection.\arabic{theorem}}[section]
\newtheorem{Corollary}[thm]{Corollary}
\newtheorem{Lemma}[thm]{Lemma}
\newtheorem{Proposition}[thm]{Proposition}

\newtheorem{Theorem}[thm]{Theorem}
\newtheorem{Definition}[thm]{Definition}
\newtheorem{Example}[thm]{Example}
\newtheorem{Remark}[thm]{Remark}

\numberwithin{equation}{section} 

\journal{-------------}
\begin{document}
\begin{frontmatter}

\title{Trace-Controlled Operator Selection and Discretization of Continuous Frames in Hilbert \(C^*\)-Modules} 

\author[aff1]{Hicham Tarif\corref{cor1}}
\ead{hicham.tarif@uir.ac.ma}


\affiliation[aff1]{
  organization={TICLab, College of Engineering and Architecture, International University of Rabat},
  addressline={Sala Al Jadida},
  postcode={11000},
  country={Morocco}
}

\cortext[cor1]{Corresponding author}


\begin{abstract}

The discretization of continuous frames in Hilbert spaces is well established, while the corresponding problem for Hilbert $C^*$-modules presents additional difficulties arising from the $C^*$-algebra-valued inner product and the structure of compact operators on Hilbert modules. In this paper, we develop an operator-selection framework for the discretization of continuous frames in Hilbert $C^*$-modules over finite-dimensional $C^*$-algebras. Our main result is a trace-controlled selection theorem for positive compact operators acting on countably generated Hilbert $C^*$-modules over finite-dimensional $C^*$-algebras. In particular, we establish a module version of the binary selection principle for positive trace-controlled compact operators. The proof combines finite-rank approximation, the matrix representation of finite-dimensional $C^*$-algebras, and operator estimates adapted to the $C^*$-module setting. We then apply these selection results to the discretization of continuous Hilbert $C^*$-module frames. Under suitable assumptions on the underlying metric measure space, with the coefficient algebra $\mathcal A$ finite-dimensional, and assuming that every closed submodule of $\mathcal H$ is orthogonally complemented. The resulting sampling family is a Hilbert $C^*$-module frame with explicit frame bounds and controlled separation.

\end{abstract}



\begin{keyword}
$C^*-$Algebra \sep $C^*-$Hilbert module \sep Hilbert module frames \sep Binary selector\sep Hilbert module operators \sep Discretisation.



\MSC 46L08 \sep 42C15 \sep 46L05

\end{keyword}

\end{frontmatter}


\section{Introduction}

\subsection{Related work: mathematical setting}

Hilbert $C^*$-modules provide a natural framework for extending Hilbert space techniques to settings in which a possibly noncommutative $ C^*$-algebra replaces the scalar field. Since their introduction in Kaplansky~\cite{kaplansky1953modules} and subsequent systematic development~\cite{lance1995hilbert, manuilov2005hilbert}, Hilbert $C^*$-modules have become an important framework in operator algebras, operator theory, and frame theory. In particular, they provide a natural setting for studying families of vectors whose coefficients carry additional algebraic information that is not visible at the level of the underlying Hilbert space.
This perspective is particularly relevant in frame theory. A Hilbert space may admit a frame when it is regarded only as a Hilbert space, while additional structure on the space may naturally require the frame coefficients to take values in a $C^*$-algebra~\cite{han2008characterizations,xia1996vector}. Hilbert $C^*$-module frames preserve this additional structure and therefore provide a suitable framework for operator, matrix, and $C^*$-algebra-valued signal models~\cite{san2012matrix}. In this paper, we are interested in the continuous version of this theory and, more specifically, in the problem of discretizing continuous Hilbert $C^*$-module frames while preserving a prescribed geometric separation of the sampling points.

We briefly recall the basic framework.

\begin{Definition}[$C^*$-Hilbert modules~\cite{lance1995hilbert}]
Let $\mathcal A$ be a $C^*$-algebra, not necessarily unital or commutative. A \emph{pre-Hilbert $\mathcal A$-module} is a left $\mathcal A$-module $\mathcal H$ equipped with a map:
$$
\langle\cdot,\cdot\rangle_{\mathcal A}:\mathcal H\times\mathcal H\longrightarrow\mathcal A,
$$
called an $\mathcal A$-valued inner product, such that, for all $x,y,z\in\mathcal H$, $a\in\mathcal A$, and $\lambda,\mu\in\mathbb C$,
\begin{enumerate}
\item $\langle \lambda x+\mu y,z\rangle_{\mathcal A}
=\lambda\langle x,z\rangle_{\mathcal A}
+\mu\langle y,z\rangle_{\mathcal A}$;
\item $\langle ax,y\rangle_{\mathcal A}
=a\langle x,y\rangle_{\mathcal A}$;
\item $\langle y,x\rangle_{\mathcal A}
=\langle x,y\rangle_{\mathcal A}^{*}$;
\item $\langle x,x\rangle_{\mathcal A}\geq 0$ in $\mathcal A$, and
\[
\langle x,x\rangle_{\mathcal A}=0
\quad\Longleftrightarrow\quad x=0.
\]
\end{enumerate}
The induced norm is defined by
$$
\|x\|_{\mathcal H}
:=\|\langle x,x\rangle_{\mathcal A}\|_{\mathcal A}^{1/2},
\qquad x\in\mathcal H.
$$
If $\mathcal H$ is complete with respect to this norm, then $\mathcal H$ is called a \emph{Hilbert $\mathcal A$-module}.
\end{Definition}

Frames in Hilbert $C^*$-modules provide an important tool for preserving the module structure. Continuous Hilbert-module frames naturally extend this framework and allow the construction of countable discrete frames through discretization, making the discretization problem fundamental to the theory. For a countably generated Hilbert module, one defines frames as follows~\cite{ghasemi2024continuous}; see also~\cite{lfounoune2023continuous}.

\begin{Definition}
Let $(\Omega,\mu)$ be a measure space and let $\mathcal H$ be a Hilbert $\mathcal A$-module. A mapping
$$
F:\Omega\longrightarrow\mathcal H
$$
is called a \emph{continuous frame} for $\mathcal H$ if the following conditions hold:
\begin{enumerate}
\item[(i)] $F$ is weakly measurable, that is, for every $x\in\mathcal H$, the mapping
\(    \Omega\longrightarrow\mathcal A,
    \qquad
    \omega\longmapsto
    \langle x,F(\omega)\rangle_{\mathcal A}
   \)
is measurable;
\item[(ii)] there exist constants $A,B>0$ such that
\[
A\langle x,x\rangle_{\mathcal A}
\leq
\int_{\Omega}
\langle x,F(\omega)\rangle_{\mathcal A}
\langle F(\omega),x\rangle_{\mathcal A}
\,d\mu(\omega)
\leq
B\langle x,x\rangle_{\mathcal A},
\qquad x\in\mathcal H.
\]
\end{enumerate}
The constants $A$ and $B$ are called the \emph{lower} and \emph{upper frame bounds}, respectively.
If only the upper inequality in (ii) holds, then $F$ is called a \emph{continuous Bessel family} (or \emph{Bessel mapping}).
A continuous frame is called \emph{tight} if $A=B$. In particular, it is called a \emph{Parseval continuous frame} if $A=B=1$.
\end{Definition}


\subsection{Contributions}

The main contributions of this paper are as follows. First, we establish a trace-controlled selection theorem for positive compact operators acting on countably generated Hilbert $C^*$-modules over finite-dimensional $C^*$-algebras (Theorem~\ref{IT}). This provides a module analogue of the binary selection mechanism available in Hilbert spaces, with quantitative control expressed in terms of the trace of the selected operators (see definition~\ref {TC}). Second, we develop a binary selection principle for positive trace-class compact operators on Hilbert $C^*$-modules, with an error estimate governed by the same square-root trace scale as in the Hilbert-space setting (Theorem~\ref{IT2}). These results provide the operator-selection framework needed for the discretization problem (Theorem~\ref{IT2} and Corollary~\ref{mainco}).

We then apply the selection results of Theorem~\ref{IT2} and Corollary~\ref{mainco} to continuous frames in Hilbert $C^*$-modules. In particular, we obtain a uniformly discrete sampling sequence whose associated discrete frame operator approximates the continuous Hilbert-module frame operator in operator norm. As a consequence, the resulting sampling family is a Hilbert $C^*$-module frame, with explicit frame bounds and a controlled separation of the sampling points.

The paper is organized as follows. Section~\ref{S2} contains the preliminary definitions and results concerning $C^*$-algebras, Hilbert $C^*$-modules, and compact operators on Hilbert modules. Section~\ref{S3} develops the operator-selection results for positive trace-class compact operators on Hilbert $ C^*$-modules over finite-dimensional $C^*$-algebras. In particular, we establish the module versions of the binary selection results used in the subsequent discretization argument. Section~\ref{S4} applies these selection results, together with the preceding $C^*$-algebraic and geometric estimates, to the discretization of continuous frames on countably generated Hilbert $C^*$-modules. In this section, we obtain the uniform discretization result under the assumption that every closed submodule admits an orthogonal complement.

\subsection{Notations}

Throughout the paper, $\mathcal{A}$ denotes a $C^*$-algebra and $\mathcal{H}$ denotes a Hilbert $\mathcal{A}$-module. We write $\mathcal{L}_{\mathcal{A}}(\mathcal{H},\mathcal{E})$ for the space of adjointable $\mathcal{A}$-linear operators from $\mathcal{H}$ to $\mathcal{E}$. In the case $\mathcal{E}=\mathcal{H}$, we simply write $\mathcal{L}_{\mathcal{A}}(\mathcal{H}):=\mathcal{L}_{\mathcal{A}}(\mathcal{H},\mathcal{H}).$
We use $I_{\mathcal{H}}$ for the identity operator on $\mathcal{H}$ and $e_{\mathcal{A}}$ for the identity element of $\mathcal{A}$ whenever $\mathcal{A}$ is unital. We denote by $\mathcal{A}_+$ the positive cone of the $C^*$-algebra $\mathcal{A}$. The order relation on $\mathcal{A}$ is denoted by $\leq_{\mathcal A}$; whenever there is no ambiguity, we simply write $\leq$. For a countable set $I\subset \mathbb{N}$ we denot the cardinal of $I$ as $|I|$. Other notation and terminology used throughout the paper will be introduced and explained when it first appears.

\section{Preliminaries}\label{S2}
In this section, we will consider all the basic definitions and Lemmas also some known properties and theorems concerning $ C^*$-algebras and $ C^*$-Hilbert modules. 

\subsection{Hilbert \(C^*\)-modules and compact operators}
The space of Hilbert-module compact operators on a Hilbert
$\mathcal A$-module $\mathcal H$ is defined by
\[
\mathcal K_{\mathcal A}(\mathcal H)
:=
\overline{\operatorname{span}_{\mathcal{A}}
\left\{
\theta_{x,y}:x,y\in\mathcal H
\right\}},
\]
where, for $x,y,z\in\mathcal H$, the rank-one operator
$\theta_{x,y}$ is defined by
\[
\theta_{x,y}(z):=\langle z,x\rangle_{\mathcal A}y.
\]
Here, the closure is taken in the operator norm.

For our purposes, we need the following definition, which is a generalization of trace-class operators from Hilbert spaces to Hilbert module operators.

\begin{Definition}[Trace-type compact operators]\label{TC}
Let $\mathcal H$ be a countably generated Hilbert $\mathcal A$-module,
where $\mathcal A$ is a finite-dimensional $C^*$-algebra, and let
$\{e_i\}_{i\in\mathbb N}$ be an orthogonal basis of $\mathcal H$.
A positive compact operator $T\in\mathcal K_{\mathcal A}(\mathcal H)$
is said to be of trace type if
\[
\operatorname{tr}_{\{e_i\}_{i\in\mathbb N}}(T)
:=
\sum_{i\in\mathbb N}
\langle Te_i,e_i\rangle_{\mathcal A}
\]
converges in $\mathcal A$.
We denote the corresponding class of trace-type operators by
$\mathcal T_{\mathcal A}(\mathcal H)$. 

We also consider the closed subspace
\[
\mathcal{SK}_{\mathcal A}(\mathcal H)
:=
\overline{\operatorname{span}_{\mathcal{A}}
\left\{
\theta_{e_i,e_i}:i\in\mathbb N
\right\}}.
\]
\end{Definition}

If $T\in\mathcal T_{\mathcal A}(\mathcal H)$ is positive, then, in
particular, $T$ is a compact operator and hence admits approximation
in operator norm by finite-rank operators of the form
\[
T_n
=
\sum_{k=1}^{N_n}\theta_{x_k^{(n)},y_k^{(n)}},
\qquad
x_k^{(n)},y_k^{(n)}\in\mathcal H,
\]
with
\[
\|T-T_n\|_{\mathcal L_{\mathcal A}(\mathcal H)}
\longrightarrow 0.
\]
A positive compact operator have a diagonalisation corresponding to an orthonormal basis of the Hilbert module Theorem~\ref{TED}; in particular, when the corresponding diagonalization theorem applies, one obtains finite-rank positive approximants of the form
\[
T_n
=
\sum_{k=1}^{N_n}
a_k^{(n)}\theta_{x_k^{(n)},x_k^{(n)}},
\qquad
a_k^{(n)}\in\mathcal A_+.
\]
The elements $ a_k^ {(n)} $ are positive in the algebra $\mathcal{A}$ because the operator $T$ is positive; see Definition~\ref{PO}.
\begin{Definition}[Positive compact operator]\label{PO}
Let $\mathcal A$ be a finite-dimensional $C^*$-algebra and let $\mathcal E$ be a Hilbert
$\mathcal A$-module. An adjointable compact operator

$$
K\in\mathcal K_{\mathcal A}(\mathcal H)
$$

is called \emph{positive} if

$$
\langle Kx,x\rangle_{\mathcal A}\geq 0,
\qquad x\in\mathcal H.
$$

\end{Definition}

\begin{Definition}[Eigenvector and eigenvalue]
Let $K\in\mathcal K_{\mathcal A}(\mathcal E)$ be a compact operator. An element
$x\in\mathcal E$ is called an \emph{eigenvector} of $K$ if the submodule
$\mathcal A x$ generated by $x$ is projective and there exists an element
$A\in\mathcal A$ such that

$$
Kx=Ax.
$$

The element $A\in\mathcal A$ is called an \emph{eigenvalue} associated with $x$.
\end{Definition}

\begin{Theorem}[Diagonalization of compact positive operators]
\label{TED}
Let $\mathcal A$ be a finite-dimensional $C^*$-algebra and let
$\mathcal H$ be a Hilbert $\mathcal A$-module. Let $K\in\mathcal K_{\mathcal A}(\mathcal H)$
be a positive compact operator such that $\ker K=\{0\}.$
Then there exists a family $\{x_i\}_{i\in I}\subseteq\mathcal H$ and positive elements $A_i\in\mathcal A$ such that
\[
\langle x_i,x_j\rangle_{\mathcal A}
=
\delta_{ij}1_{\mathcal A},
\qquad
Kx_i=A_ix_i,
\]
and
\[
\mathcal H
=
\overline{\operatorname{span}_{\mathcal A}\{x_i:i\in I\}}.
\]
In particular, $\{x_i\}_{i\in I}$ is a complete orthonormal family of
$\mathcal A$-module eigenvectors for $K$.
\end{Theorem}

\begin{Lemma}[Order and squares of positive operators]\label{TL}
Let $\mathcal{H}$ be a Hilbert $C^*$-module over a $C^*$-algebra $\mathcal A$, and let $S\in\mathcal L_{\mathcal A}(\mathcal{H})$ be a positive operator. If $0\leq S\leq cI_\mathcal{H}$, for some $c\geq 0$, then $0\leq S^2\leq cS\leq c^2I_\mathcal{H}.$
In particular, $0\leq S^2\leq c^2I_\mathcal{H}.$
\end{Lemma}

\begin{proof}
Since $\mathcal L_{\mathcal A}(\mathcal{H})$ is a $C^*$-algebra and $0\leq S\leq cI_\mathcal{H},$
we have $cI_\mathcal{H}-S\geq 0.$
Moreover, $S$ and $cI_\mathcal{H}-S$ commute, since $S(cI_\mathcal{H}-S)=cS-S^2=(cI_\mathcal{H}-S)S.$
By the standard property of $C^*$-algebras that the product of two commuting positive elements is positive, i.e., function calculus in $C^*$-algebras, we obtain
\[
S(cI_\mathcal{H}-S)\geq 0.
\]
Hence $
cS-S^2\geq 0,$ and consequently $S^2\leq cS.$

Since $S\geq0$ and $c\geq0$, we have $cS\geq0.$
Furthermore, multiplying the inequality $S\leq cI_\mathcal{H}$ by the non-negative scalar $c$ gives
\[
cS\leq c^2I_\mathcal{H}.
\]
Finally, since $S\geq0$, we have $S^2\geq0$.

Therefore,
\[
0\leq S^2\leq cS\leq c^2I_\mathcal{H}.
\]
\end{proof}

\begin{Remark}[Functional calculus]
The same conclusion can be obtained using the continuous functional
calculus in the $C^*$-algebra $\mathcal L_{\mathcal A}(\mathcal{H})$. Indeed, since $S\geq0$, the operator $S$ is self-adjoint, and hence the
continuous functional calculus applies to $S$. The inequality
\[
0\leq S\leq cI_\mathcal{H},\;\text{implies}\;
\sigma(S)\subseteq[0,c].
\]
Consider the continuous function
\[
f(t)=t(c-t),\qquad t\in[0,c].
\]
Since
\[
f(t)\geq0,\qquad t\in[0,c],
\]
the positivity-preserving property of the continuous functional calculus
gives
\[
f(S)=S(cI_\mathcal{H}-S)=cS-S^2\geq0.
\]
Thus $S^2\leq cS$. Combining this with $S\leq cI_E$ yields
\[
S^2\leq cS\leq c^2I_\mathcal{H}.
\]
Hence the result does not rely on the Hilbert space structure. It follows from the $C^*$-algebra structure of $\mathcal L_{\mathcal A}(\mathcal{H})$, where the continuous functional calculus is available for every normal element.
\end{Remark}

\subsection{\(C^*\)-algebras trace map}

We begin this section with technical lemmas concerning the order in a $C^*$-algebra $\mathcal{A}$. This lemma is well known (see~\cite{lance1995hilbert,lindenhovius2015classifying,takesaki2002operator}), but since we use it frequently in the main results of this paper, we state and prove it here.

\begin{Lemma}[\cite{lindenhovius2015classifying}]
Let $\mathcal{A}$ be a $C^*$-algebra and $a,b\in\mathcal{A}$, we have:
\begin{enumerate}
    \item $(a+b)^*(a+b)\leq 2(a^*a+b^*b)$.\label{1}
    \item $\|aa^*-bb^*\|\leq (\|a\|+\|b\|)\|a-b\|$.\label{2}
    \item If $0\leq a\leq b$ then $\|a\|\leq \|b\|$.\label{3}
\end{enumerate}
\end{Lemma}
\begin{proof}
    For~\ref{1} we have $(a+b)^*(a+b)=a^*a+a^*b+b^*a+b^*b$ since $(a-b)^*(a-b)\geq 0$ we have $a^*b+b^*a\leq a^*a+b^*b$ This implies~\ref{1}.
    As for the prove of~\ref{2} we have $aa^*-bb^*=(a-b)a^*+b(a^*-b^*)$ Hence \begin{align*}
        \|aa^*-bb^*\|&=\|(a-b)a^*+b(a^*-b^*)\|\\
        &\leq \|a-b\|\|a^*\|+\|b\|\|a^*-b^*\|\\
        &=(\|a\|+\|b\|)\|a-b\|.
    \end{align*}
    Suppose now that $0\leq a \leq b$. We have for all positive elements $a\in\mathcal{A}_+$ the following characterization $\|a\|=\inf\{\lambda>0\;:\;a\leq \lambda e_{\mathcal{A}}\}$. Since $a\leq b\leq\|b\|e_{\mathcal{A}}$ this implies that $\|a\|\leq \|b\|$.
\end{proof}

The proof of Theorem~\ref{IT} relies on the following isometric characterization of finite-dimensional $C^*$-algebras.

\begin{Theorem}[Structure theorem for finite-dimensional $C^*$-algebras~\cite{lindenhovius2015classifying}]\label{fds}
Let $\mathcal A$ be a finite-dimensional $C^*$-algebra. Then there exist positive integers $k,d_1,\ldots,d_k$ such that
\[
\mathcal A \cong \bigoplus_{i=1}^k M_{d_i}(\mathbb C)
\]
as $C^*$-algebras. Moreover, the integer $k$ and the multi-set
$\{d_1,\ldots,d_k\}$ are uniquely determined by $\mathcal A$.
\end{Theorem}

\begin{Definition}[Faithful finite trace~\cite{takesaki2002operator}]
Let $\mathcal A$ be a $C^*$-algebra. A \emph{trace} on $\mathcal A$ is a positive linear functional
$$
\tau_{\mathcal A}:\mathcal A\longrightarrow\mathbb C
$$
such that
$$
\tau_{\mathcal A}(ab)=\tau_{\mathcal A}(ba),
\qquad a,b\in\mathcal A.
$$
The trace $\tau_{\mathcal A}$ is said to be \emph{faithful} if
$$
\tau_{\mathcal A}(a^*a)=0
\quad\Longrightarrow\quad
a=0.
$$

If $\mathcal A$ is unital, the trace is said to be \emph{finite} if
$$
\tau_{\mathcal A}(1_{\mathcal A})<\infty.
$$
A faithful finite trace is a trace satisfying both properties above.
\end{Definition}

\begin{Definition}[Equivalence of traces]
Let $\mathcal{A}$ and $\mathcal{B}$ be $C^*$-algebras equipped with traces $\tau_{\mathcal{A}}$ and $\tau_{\mathcal{B}}$, respectively. We say that $\tau_{\mathcal{A}}$ and $\tau_{\mathcal{B}}$ are \emph{equivalent}, and write
$$
\tau_{\mathcal{A}}\equiv\tau_{\mathcal{B}},
$$
if there exists a $*$-isomorphism $
\phi:\mathcal{A}\longrightarrow\mathcal{B}$
such that
$$
\tau_{\mathcal{B}}(b)
=
\tau_{\mathcal{A}}\bigl(\phi^{-1}(b)\bigr),
\qquad b\in\mathcal{B}.
$$
\end{Definition}

\begin{Definition}
    For trace-class operators (definition~\ref{TC}) we consider the trace map $\operatorname{tr}_{\mathcal{A}}$ defined as $$\operatorname{tr}_{\mathcal{A}}(T)=\sum_{i\in\mathbb N}
\tau_{\mathcal{A}}(\langle Te_i,e_i\rangle_{\mathcal A}).$$
\end{Definition}

\begin{Example}\label{E}
Let
$$
\mathcal{A}
=
\bigoplus_{r=1}^{m} M_{d_r}(\mathbb{C})
$$

be a finite-dimensional $C^*$-algebra. A trace on $\mathcal{A}$ is of the form

$$
\tau_{\mathcal{A}}(a_1,\ldots,a_m)
=
\sum_{r=1}^{m} c_r\,\operatorname{Tr}(a_r),
\qquad
c_r\geq 0,
$$

where $\operatorname{Tr}$ denotes the usual matrix trace on
$M_{d_r}(\mathbb{C})$. Moreover, $\tau_{\mathcal{A}}$ is faithful if and only if

$$
c_r>0,\qquad r=1,\ldots,m.
$$

In particular, if $\mathcal{B}$ is a finite dimension $C^*-$algebra such that $\mathcal{B}$ is $*$-isomorphic to $\mathcal{A}$ then there exist a $*$-isomorphism $\phi:\mathcal{A}\to\mathcal{B}$ and we can define a trace on $\mathcal{B}$ as $$
\tau_{\mathcal{B}}\circ\phi=\tau_{\mathcal{A}}.
$$
For example, the canonical trace on $\mathcal{A}$ is given by

$$
\tau_{\mathcal{A}}(a_1,\ldots,a_m)
=
\sum_{r=1}^{m}\operatorname{Tr}(a_r).
$$
\end{Example}
We will use this example frequently in the next sections.

\subsection{Selector lammas}

\begin{Definition}[Binary selectors~\cite{bownik2024selector}]
Let $I$ be a countable index set, and let $\mathcal J=\{J_k\}_{k\in K}$ be a partition of $I$ such that $\#J_k=2,\quad k\in K.$

A pair of sets $(I_0,I_1)$ is called a \emph{binary selector of order $1$}
associated with the partition $\mathcal J$ if
$$
I=I_0\cup I_1,\;\text{and}\;
\#(I_0\cap J_k)
=
\#(I_1\cap J_k)
=
1,
\quad k\in K.
$$
Thus, each two-element set $J_k$ contributes exactly one element to each of $I_0$ and $I_1$. More generally, let $N\geq2$. Assume that binary selectors of order $N-1$ have been constructed, denoted by
$$
\{I_b\}_{b\in\{0,1\}^{N-1}},\:
\text{where}\:
I=\bigcup_{b\in\{0,1\}^{N-1}}I_b.
$$
For each $b\in{0,1}^{N-1}$, let $\mathcal J_b=\{J_{b,k}\}_{k\in K}
$ be a partition of $I_b$ satisfying $\#J_{b,k}=2,
\qquad k\in K.
$
A family $
\{I_{b0},I_{b1}\}_{b\in\{0,1\}^{N-1}}
$ is called a \emph{binary selector of order $N$} if, for every $b\in{0,1}^{N-1}$, we have $I_b=I_{b0}\cup I_{b1}$, and $\#(I_{b0}\cap J_{b,k})=\#(I_{b1}\cap J_{b,k})=1,\quad k\in K.$
Equivalently, each two-element set $J_{b,k}$ is split into exactly one element belonging to $I_{b0}$ and one element belonging to $I_{b1}$.
\end{Definition}

The following lemma appears in the work of Bownik~\cite{bownik2024selector} in the form of a remark~\cite[Remark 2.15]{bownik2024selector}. For completeness, we state it here as a separate lemma and provide a proof. 



\begin{Lemma}\label{be}
Let $\varepsilon>0$, $d_1,\ldots,d_k\in\mathbb N$, and $D=\sum_{j=1}^{k}d_j$. Suppose that
$X_1,\ldots,X_m$ are jointly independent positive semi-definite random matrices taking finitely many values and having the block diagonal form
\[
X_i=
\operatorname{diag}\!\left(
X_i^{(1)},X_i^{(2)},\ldots,X_i^{(k)}
\right),
\qquad i=1,\ldots,m,
\]
where
\[
X_i^{(j)}\in M_{d_j}(\mathbb C)_+,
\qquad j=1,\ldots,k.
\]
Assume that, for every $j=1,\ldots,k$,
\[
\sum_{i=1}^{m}\mathbb E\!\left[X_i^{(j)}\right]
    \leq I_{d_j},\:\:
\text{and}\quad
\tau_{M_{d_j}(\mathbb C)}\left(X_i^{(j)}\right)=\operatorname{Tr}\!\left(X_i^{(j)}\right)
    \leq \varepsilon
\]
almost surely for every $i=1,\ldots,m$.

Then
\[
\mathbb P\!\left(
\|\sum_{i=1}^{m}X_i\|
    \leq
1+2\sqrt{\varepsilon k}+\varepsilon
\right)>0.
\]
\end{Lemma}
\begin{proof}
 To prove this lemma, it suffices to note the following $*-$homomorphism and apply the lemma for matrices of the same size~\cite [Theorem~1.2]{bownik2024selector}.
 
\begin{Lemma}\label{lem:matrix_embedding}
Let $d_1,d\in\mathbb N$ with $d_1\leq d$. Define the map
\[
\iota:M_{d_1}(\mathbb C)\longrightarrow M_d(\mathbb C)
\]
by
\[
\iota(A)=
\begin{pmatrix}
A & 0\\
0 & 0
\end{pmatrix},
\qquad A\in M_{d_1}(\mathbb C).
\]
Then the following assertions hold.

\begin{enumerate}
    \item $\iota$ is an injective unital $*$-homomorphism onto its image.
    
    \item $\iota$ is isometric, that is,
    \[
    \|\iota(A)\|=\|A\|,
    \qquad
    A\in M_{d_1}(\mathbb C).
    \]
    
    \item $\iota$ preserves positivity, namely,
    \[
    A\geq0
    \quad\Longleftrightarrow\quad
    \iota(A)\geq0.
    \]
    
    \item $\iota$ preserves the operator order. In particular,
    \[
    A\leq B
    \quad\Longleftrightarrow\quad
    \iota(A)\leq\iota(B),
    \qquad
    A,B\in M_{d_1}(\mathbb C).
    \]
    
    \item $\iota$ preserves the trace, that is,
    \[
    \operatorname{Tr}(\iota(A))
    =
    \operatorname{Tr}(A),
    \qquad
    A\in M_{d_1}(\mathbb C).
    \]
\end{enumerate}
\end{Lemma}

Since
\[
\iota(AB)
=
\begin{pmatrix}
AB&0\\
0&0
\end{pmatrix}
=
\begin{pmatrix}
A&0\\
0&0
\end{pmatrix}
\begin{pmatrix}
B&0\\
0&0
\end{pmatrix}
=
\iota(A)\iota(B),
\]
and
\[
\iota(A^*)
=
\begin{pmatrix}
A^*&0\\
0&0
\end{pmatrix}
=
\iota(A)^*,
\]
the map $\iota$ is an injective $*$-homomorphism.

Moreover,
\[
\iota(A)^*\iota(A)
=
\begin{pmatrix}
A^*A&0\\
0&0
\end{pmatrix},
\]
whose spectrum consists of the spectrum of $A^*A$ together with additional
zeros. Hence
\[
\|\iota(A)\|^2
=
\|\iota(A)^*\iota(A)\|
=
\|A^*A\|
=
\|A\|^2,
\]
which proves that $\iota$ is isometric.

If $A\geq0$, then for every vector
$
\binom{x}{y}
\in\mathbb C^{d_1}\oplus\mathbb C^{d-d_1},
$
\[
\left\langle
\iota(A)
\binom{x}{y},
\binom{x}{y}
\right\rangle
=
\langle Ax,x\rangle
\geq0,
\]
so $\iota(A)\geq0$. The converse follows by restricting to vectors of the
form $\binom{x}{0}$.

The preservation of the operator order follows immediately from positivity,
since
\[
A\leq B
\iff
B-A\geq0
\iff
\iota(B-A)\geq0
\iff
\iota(A)\leq\iota(B).
\]

Finally,
\[
\operatorname{Tr}(\iota(A))
=
\operatorname{Tr}
\begin{pmatrix}
A&0\\
0&0
\end{pmatrix}
=
\operatorname{Tr}(A),
\]
since the additional diagonal entries are zero.
\end{proof}

A direct consequence of this lemma is the following.

\begin{Lemma}\label{psd1}
Let $\mathcal A$ be a finite-dimensional unital $C^*$-algebra. By the
finite-dimensional structure theorem, there exist integers
$k,d_1,\ldots,d_k\in\mathbb N$ such that
\[
\mathcal A\simeq\bigoplus_{j=1}^{k}M_{d_j}(\mathbb C).
\]
In particular, $k$ denotes the number of matrix blocks in this
decomposition.

Let $\{X_i\}_{i\in I}$ be a finite family of $\mathcal A_+$-valued random
variables satisfying
\[
\sum_{i\in I}\mathbb E[X_i]\leq I_{\mathcal A}
\]
and
\[
\tau_{\mathcal A}(X_i)\leq\varepsilon
\qquad\text{almost surely for every }i\in I.
\]
Then
\[
\mathbb P\left(
\left\|\sum_{i\in I}X_i\right\|_{\mathcal A}
\leq
1+2\sqrt{\varepsilon k}+\varepsilon
\right)>0.
\]
\end{Lemma}

\begin{proof}
By the finite-dimensional structure theorem, fix a $*$-isomorphism
\[
\Phi:\mathcal A\longrightarrow
\bigoplus_{j=1}^{k}M_{d_j}(\mathbb C).
\]
Thus, for each $i\in I$,
\[
\Phi(X_i)
=
\operatorname{diag}
\bigl(
X_i^{(1)},\ldots,X_i^{(k)}
\bigr),
\]
where
\[
X_i^{(j)}\in M_{d_j}(\mathbb C)_+.
\]

Since $\Phi$ is a $*$-isomorphism (it is injective; for the surjectivity we take the image as the output set), it preserves positivity, order, and
the $C^*$-norm.
The trace on
\(\bigoplus_{j=1}^{k}M_{d_j}(\mathbb C)\) is understood as the trace transported from \(\mathcal A\) through \(\Phi\), namely
\[
\tau_{\mathcal A}
=
\tau_{\oplus_{j=1}^{k}M_{d_j}(\mathbb C)}\circ\Phi.
\]

Consequently, for every $X\in\mathcal A_+$,
\[
\tau_{\oplus_{j=1}^{k}M_{d_j}(\mathbb C)}(\Phi(X))
=
\tau_{\mathcal A}(X).
\]

Consequently,
\[
\sum_{i\in I}\mathbb E[X_i^{(j)}]\leq I_{d_j},
\qquad
\tau_{\mathcal A}(X_i)\leq\varepsilon
\]
for every $j=1,\ldots,k$.

If the block sizes $d_j$ are not equal, we embed each
$M_{d_j}(\mathbb C)$ into $M_d(\mathbb C)$, where
\[
d:=\max_{1\leq j\leq k}d_j,
\]
by padding with zero rows and columns. This yields positive block matrices
\[
\widetilde X_i
=
\operatorname{diag}
\bigl(
\widetilde X_i^{(1)},\ldots,
\widetilde X_i^{(k)}
\bigr)
\in M_{dk}(\mathbb C)_+,
\]
with
\[
\|\widetilde X_i\|=\|X_i\|_{\mathcal A}
\quad
\text{and}\quad
\sum_{i\in I}\mathbb E[\widetilde X_i]
\leq I_{dk}.
\]

The hypotheses of Lemma~\ref{be} are therefore satisfied for the
$k$ diagonal blocks. Hence
\[
\mathbb P\left(
\left\|\sum_{i\in I}\widetilde X_i\right\|
\leq
1+2\sqrt{\varepsilon k}+\varepsilon
\right)>0.
\]
Finally, the block-diagonal realization preserves the norm, so
\[
\left\|\sum_{i\in I}\widetilde X_i\right\|
=
\left\|\sum_{i\in I}X_i\right\|_{\mathcal A}.
\]
The desired conclusion follows.
\end{proof}



This lemma provides the one-step binary selection estimate, but its $\sqrt{\varepsilon k}$ dependence is not suitable for the trace-controlled setting. We therefore establish an alternative estimate, inspired by~\cite{bownik2024selector,bownik2026akemann}, adapted to finite-trace operators, which we use in the subsequent selection theorem.

We use two different estimates for the norm of sums of matrices, given in Lemma~\ref{psd1} and Lemma~\ref{psd}. The first follows from the binary selection principle, while the second provides the estimate required for the trace-controlled selection argument.

\begin{Lemma}\label{psd}
Let $\mathcal A$ be a finite-dimensional unital $C^*$-algebra, and let $X_k$, $k\in\{1,\ldots,m\}$, be $\mathcal A_+$-valued random variables
such that
\[
\sum_{k=1}^{m}\mathbb E[X_k]\leq I_{\mathcal A}
\]
and
\[
\operatorname{tr}_{\mathcal A}(X_k)\leq\varepsilon
\qquad\text{almost surely for every }k\in\{1,\ldots,m\}.
\]
Then
\[
\mathbb P\left(
\left\|\sum_{k=1}^{m}X_k\right\|_{\mathcal A}
\leq
(1+\sqrt{\varepsilon})^2
\right)>0.
\]
\end{Lemma}

\begin{proof}
Let
\[
\Phi:\mathcal A\longrightarrow
\bigoplus_{j=1}^{r}M_{d_j}(\mathbb C)
\]
be a $*$-isomorphism, where $r$ is the number of simple matrix blocks in
the finite-dimensional decomposition of $\mathcal A$. By the preceding
discussion, the trace is transported through $\Phi$, so that
\[
\tau_{\oplus_{j=1}^{r}M_{d_j}(\mathbb C)}
\bigl(\Phi(X_k)\bigr)
=
\operatorname{tr}_{\mathcal A}(X_k).
\]
Using the decomposition
\[
\mathcal A\simeq\bigoplus_{j=1}^{r}M_{d_j}(\mathbb C),
\]
we regard each element $\Phi(X_k)$ as a positive matrix in the corresponding
finite-dimensional matrix algebra, after embedding the blocks into a common
matrix algebra $M_D(\mathbb C)$ by zero-padding if necessary, where
\[
D=\sum_{j=1}^{r}d_j.
\]
Thus, the elements $\Phi(X_k)$ are viewed as whole $D\times D$ positive
semidefinite matrices, rather than merely as block-diagonal elements of the
direct-sum algebra. The trace and expectation assumptions are preserved
under this realization. We may therefore apply Theorem~\ref{fds} (or
Theorem~\cite[Theorem~1.1]{bownik2024selector}) to the resulting family of positive
semidefinite matrices, obtaining
\[
\mathbb P\left(
\left\|\sum_{k=1}^{m}\Phi(X_k)\right\|_{M_D(\mathbb C)}
\leq
(1+\sqrt{\varepsilon})^2
\right)>0.
\]
Since the matrix realization is isometric, this yields the desired
estimate in $\mathcal A$.
\[
\mathbb P\left(
\left\|\sum_{k=1}^{m}\Phi(X_k)\right\|
\leq
(1+\sqrt{\varepsilon})^2
\right)>0.
\]
Since $\Phi$ is a $*$-isomorphism and hence isometric,
\[
\left\|\sum_{k=1}^{m}\Phi(X_k)\right\|
=
\left\|\sum_{k=1}^{m}X_k\right\|_{\mathcal A},
\]
which proves the result.
\end{proof}

\subsection{Index set properties and assumptions}


\begin{Definition}[Standing assumptions on the measure space~\cite{bownik2026uniform}]\label{AD}
Let $(X,d,\mu)$ be a metric measure space. We impose the following assumptions.
\begin{enumerate}
\item[\textnormal{(A1)}] The measure $\mu$ is infinite:
\(    \mu(X)=\infty.
   \)
Moreover, every metric ball has positive measure, that is,
\(    0<\mu(B_r(x))<\infty,
    \qquad x\in X,\quad r>0.
   \)
\item[\textnormal{(A2)}] \emph{Small-scale doubling condition.}
There exist constants $R_A>0$ and $C_D\geq 1$ such that
\[
\mu(B_{2r}(x))
\leq
C_D\,\mu(B_r(x)),
\qquad x\in X,\quad 0<r\leq R_A.
\]
We refer to $R_A$ as the \emph{doubling scale} and to $C_D$ as the
corresponding \emph{doubling constant}.

\item[\textnormal{(A3)}] \emph{Small-scale upper Ahlfors regularity.}
There exist constants $C_A>0$ and $\gamma>0$ such that
\[
\mu(B_r(x))
\leq
C_A r^\gamma,
\qquad x\in X,\quad 0<r\leq R_A.
\]
We call $\gamma$ an \emph{upper Ahlfors exponent} and $C_A$ the
corresponding \emph{upper Ahlfors constant}.
\end{enumerate}
\end{Definition}
Throughout the paper, $(X,d,\mu)$ denotes a $\sigma$-finite nonatomic metric measure space satisfying Assumptions~(A1)--(A3).
\begin{Example}[A matrix-valued metric measure space]\label{E1}
Let

$$
X=M_d(\mathbb R)_+\times\mathbb R^d,
$$

where $M_d(\mathbb R)_+$ denotes the cone of positive semidefinite
$d\times d$ real matrices. Equip $M_d(\mathbb R)$ with the Frobenius norm

$$
\|A\|_{\mathrm F}
=
\left(\sum_{i,j=1}^d |A_{ij}|^2\right)^{1/2},
$$

and define a metric on $X$ by

$$
\rho\big((A,x),(B,y)\big)
=
\left(
\|A-B\|_{\mathrm F}^2+\|x-y\|_2^2
\right)^{1/2}.
$$

Let $\mu$ be the product of the Lebesgue measure on the vector space
$M_d(\mathbb R)_{\mathrm{sym}}$ and the Lebesgue measure on $\mathbb R^d$,
restricted to $M_d(\mathbb R)_+\times\mathbb R^d$.

Then $(X,\rho,\mu)$ is a metric measure space with
$$
\mu(X)=\infty.
$$
Moreover, every nonempty metric ball has positive measure.
Since $M_d(\mathbb R)_+$ is a closed cone with nonempty interior in $M_d(\mathbb R)_{\mathrm{sym}}$, the measure of sufficiently small balls centered at points of $M_d(\mathbb R)_+$ is comparable to the corresponding Euclidean volume. Consequently, at small scales,
$$
\mu(B_r((A,x)))\lesssim r^{\gamma},
$$
where
$$
\gamma
=
\dim M_d(\mathbb R)_{\mathrm{sym}}+d
=
\frac{d(d+1)}{2}+d
=
\frac{d(d+3)}{2}.
$$
In particular, the space satisfies a local upper Ahlfors regularity
condition. It also satisfies a local doubling condition.
\end{Example}

We will return to this example later in section~\ref{S4} for multi-dimension matrices valued frames discritisation (see example~\ref{EF}).


\section{Selector form for injective positive compact Hilbert module operators}\label{S3}

Using Lemma~\ref{psd}, we can show a Hilbert module version of Theorem~3.1 in \cite[Theorem~3.1]{bownik2026uniform}, which is one of this paper's contributions.

\begin{Theorem}[Hilbert module version of Theorem 3.1~\cite{bownik2024selector}]\label{IT}
Let $\mathcal A$ be a finite-dimensional unital $C^*$-algebra and let
$\mathcal H$ be a countably generated Hilbert $\mathcal A$-module Where every closed submodule has an orthogonal complement.
Let $\varepsilon>0$ and suppose that
$\{T_i\}_{i\in I}\subset\mathcal K(\mathcal H)_+$
is a family of positive compact adjointable injective operators satisfying
\begin{equation}\label{eq:module-assumption}
\sum_{i\in I}T_i\le I_{\mathcal H},
\qquad
\operatorname{tr}_{\mathcal{A}}(T_i)\le\varepsilon,
\quad i\in I.
\end{equation}
Let $\{J_k\}_{k\in K}$ be a collection of pairwise disjoint subsets of $I$
such that
\[
|J_k|\ge r,
\qquad k\in K.
\]

Then there exists a selector
\(
J\subset\bigcup_{k\in K}J_k
\) satisfying
\(
|J\cap J_k|=1,
\: k\in K,
\)
such that
\[
\sum_{i\in J}T_i
\le
\left(\frac1{\sqrt r}+\sqrt{\varepsilon}\right)^2
I_{\mathcal H}.
\]
\end{Theorem}

\begin{proof}


We first consider the case where the index set $I$ is finite. Once the estimate has been established in the case where \(I\) is finite, the general case where \(I\) is countable follows by the same limiting argument as in Theorem~\cite[Theorem 3.1]{bownik2024selector}; see also~\cite{bownik2026akemann}.

For each $i\in I$, let $\{T_i^{(n)}\}_{n\in\mathbb N}$ be a sequence of
positive finite-rank adjointable operators such that 
\begin{equation}\label{eq:approximation}
0\leq T_i^{(n)}\leq T_i,
\qquad
\lim_{n\to\infty}\|T_i-T_i^{(n)}\|=0.
\end{equation}

Indeed, since the operators $T_i$ are self-adjoint, injective, positive Hilbert module operators, there exists an orthonormal basis $\{x_i^n\}_{n\in\mathbb{N}}$ such that $T_i=\sum_{n=1}^N a_i^n\theta_{x_i^n,x_i^n}$ where $a_i^n$ are the corresponding eigenvalues (Theorem~\ref{TED}).


Fix $n\in\mathbb N$. Since each $T_i^{(n)}$ is of finite rank, there exists
a finitely generated free Hilbert $\mathcal A$-submodule
$\mathcal F_n\subseteq\mathcal E$
such that every operator $T_i^{(n)}$ acts trivially on
$\mathcal F_n^\perp$.
After choosing a free basis of $\mathcal F_n$, we identify
$\mathcal F_n\simeq\mathcal A^{d_n}$.
Consequently,
\[
T_i^{(n)}\in M_{d_n}(\mathcal A)_+,
\qquad i\in I.
\]



For each \(k\), choose a subset \(J_k'\subset J_k\) with \(|J_k'|=r\). Since the sets \(J_k'\) are still pairwise disjoint, it suffices to construct a selector \(J\) such that $|J\cap J_k'|=1
\;\text{for every }k.$
Such a selector automatically satisfies
\[
|J\cap J_k|=1
\qquad\text{for every }k.
\]
Thus, replacing \(J_k\) by \(J_k'\), we may assume throughout the proof that \(|J_k|=r\).

For each \(k\in K\), define a positive random element
\[
X_k\in M_r\big(M_{d_n}(\mathcal A)\big)_+
\]
which takes the \(r\) values
\[
\begin{pmatrix}
rT_{i_1}^{(n)} & 0 & \cdots & 0\\
0 & 0 & \cdots & 0\\
\vdots & \vdots & \ddots & \vdots\\
0 & 0 & \cdots & 0
\end{pmatrix},
\quad\ldots,\quad
\begin{pmatrix}
0 & 0 & \cdots & 0\\
0 & 0 & \cdots & 0\\
\vdots & \vdots & \ddots & \vdots\\
0 & 0 & \cdots & rT_{i_r}^{(n)}
\end{pmatrix},
\]
each with probability \(1/r\), where
\[
J_k=\{i_1,\ldots,i_r\}.
\]
Thus, in each realization, exactly one diagonal block is nonzero. The random elements \(\{X_k\}_{k\in K}\) are jointly independent by construction, since the random choice of the nonzero diagonal block is made independently for each \(k\).

Moreover,
\[
\mathbb E [X_k]
=
\operatorname{diag}\left(
\sum_{i\in J_k}T_i^{(n)},
\,\ldots,\,
\sum_{i\in J_k}T_i^{(n)}
\right),
\]
and hence
\[
\sum_{k\in K}\mathbb E [X_k]
=
\operatorname{diag}\left(
\sum_{i\in I}T_i^{(n)},
\,\ldots,\,
\sum_{i\in I}T_i^{(n)}
\right)
\le I.
\]
Furthermore, each realization satisfies
\[
\operatorname{tr}_{\mathcal A}(X_k)
=
r\,\operatorname{tr}_{\mathcal A}(T_i^{(n)})
\le r\varepsilon .
\]

Applying lemma~\ref{psd} yields an outcome satisfying
\[
\sum_{k\in K}X_k
\leq
\left(1+\sqrt{r\varepsilon}\right)^2
I_{\mathcal H}.
\]

Hence there exists a selector
$J\subseteq I$
satisfying
\[
|J\cap J_k|=1,
\qquad k\in K,
\]
such that
\begin{equation}\label{OE}
    \sum_{i\in J}T_i^{(n)}
\leq
\left(\frac1{\sqrt r}+\sqrt\varepsilon\right)^2
I_{\mathcal H}.
\end{equation}

For each $k\in K$, the random variable $X_k$ is defined by choosing uniformly one operator
$rT_i^{(n)}$ with $i\in J_k$. Consequently, an outcome of the random experiment is equivalent
to choosing, for each $k\in K$, an index $i_k\in J_k.$
Hence every outcome determines a subset
\[
J=\{\,i_k:\;k\in K\,\}\subseteq\bigcup_{k\in K}J_k.
\]
Since exactly one index is selected from each set $J_k$, the resulting subset $J$ satisfies
\begin{equation}\label{3.2}
    |J\cap J_k|=1,\qquad \forall\,k\in K,
\end{equation}

which is precisely condition~\ref{3.2}. Therefore, the positive-probability event provided by Lemma~\ref{be} yields the existence of a selector $J$ satisfying~\ref{3.2} together with the desired operator estimate~\ref{OE}.

Since $I$ is finite, there are only finitely many selectors.
Therefore, by the pigeonhole principle, one selector $J$ satisfies the preceding estimate for infinitely many values of $n$.
Passing to the limit in \eqref{eq:approximation} gives
\[
\sum_{i\in J}T_i
\leq
\left(\frac1{\sqrt r}+\sqrt\varepsilon\right)^2
I_{\mathcal H},
\]
which completes the proof in the finite case.
\end{proof}


\begin{Theorem}\label{IT2}
Let $I$ be a countable set and let $\varepsilon>0$. Suppose that
$\{T_i\}_{i\in I}$ is a family of positive trace-class operators on a
countably generated Hilbert module (have the same conditions as the last theorem~\ref{IT}) $\mathcal H$ satisfying
\begin{equation}\label{eq:3.7}
T:=\sum_{i\in I}T_i\leq I_{\mathcal H}
\qquad\text{and}\qquad
\operatorname{tr}_{\mathcal{A}}(T_i)\leq\varepsilon,
\quad i\in I.
\end{equation}
Let $\{J_k\}_{k\in K}$ be a partition of $I$ such that

$$
\#J_k=2,
\qquad k\in K.
$$

Then there exists a selector

$$
J\subset\bigcup_{k\in K}J_k
$$

satisfying 
\begin{equation}\label{(3.2)}
    |J\cap J_k|=1,
\qquad k\in K,
\end{equation}
such that
\begin{equation}\label{eq:3.8}
\left\|
\sum_{i\in J}T_i-\frac{1}{2}T
\right\|,
\qquad
\left\|
\sum_{i\in I\setminus J}T_i-\frac{1}{2}T
\right\|
\leq
2\sqrt{\varepsilon}+\varepsilon.
\end{equation}
\end{Theorem}
\begin{proof}
We first assume that $I$ is finite and that each $T_i$ has finite
rank. For every $k\in K$, write

$$
J_k=\{i_k,i_k'\}.
$$

Since each $T_i$ is a positive and finite rank self-adjoint operator, we may write, with respect to a finite orthonormal family in $\mathcal H$,
$$
T_i=\sum_{r=1}^{m_i}a_{i,r}\,
\theta_{x_{i,r},x_{i,r}},
\qquad a_{i,r}\in\mathcal A_+.
$$

After fixing a faithful finite-dimensional representation

$$
\rho:\mathcal A\longrightarrow
\bigoplus_{\nu=1}^s M_{d_\nu}(\mathbb C),
$$

the operators $T_i$ are therefore encoded by positive finite-dimensional
block matrices whose entries are the coefficients $\rho(a_{i,r})$.
We identify these matrices with their block-diagonal realizations in a
common matrix algebra, padding by zero blocks when necessary. This
identification preserves positivity and the operator norm.

For each $k\in K$, consider the two-valued positive matrix

$$
X_k=
\begin{cases}
2\,\rho(T_{i_k})\oplus 2\,\rho(T_{i_k'}),&
\text{with probability }1/2,\\[2mm]
2\,\rho(T_{i_k'})\oplus 2\,\rho(T_{i_k}),&
\text{with probability }1/2.
\end{cases}
$$

Equivalently, $X_k$ is obtained by placing the two possible choices of the pair $\{i_k,i_k'\}$ on the diagonal. The family $\{X_k\}_{k\in K}$ is independent and

$$
\mathbb E [X_k]
=
\rho(T_{i_k}+T_{i_k'}).
$$

Consequently,

$$
\sum_{k\in K}\mathbb E [X_k]
=
\rho\!\left(\sum_{i\in I}T_i\right)
=
\rho(T)
\leq I.
$$
Moreover, by the trace assumption,
$$
\operatorname{tr}_{\mathcal{A}}(X_k^{(r)})
\leq 2\varepsilon,
\qquad r=1,2.
$$
We now apply the finite-dimensional positive-matrix selection estimate established above to the family $\{X_k\}_{k\in K}$, together with the deterministic positive matrices corresponding to the complementary operator $I-\rho(T)$.
More precisely, choose a finite collection of deterministic positive semidefinite matrices \(\{Y_q\}_{q\in K'}\) of the form
\[
Y_q=Y_q'\oplus Y_q',
\qquad q\in K',
\]
such that
\[
\operatorname{tr}_{\mathcal{A}}(Y_q')\leq 2\varepsilon
\quad\text{and}\quad
\sum_{q\in K'}Y_q=I-\rho(T).
\]

More precisely, choose a finite decomposition
$$
I-\rho(T)=\sum_{q\in K'}Y_q,
\qquad Y_q\geq0,
$$
with the trace bound required by the selection lemma. We obtain an
outcome for which

$$
\sum_{k\in K}X_k+\sum_{q\in K'}Y_q
\leq
(1+4\sqrt{\varepsilon}+2\varepsilon)I.
$$

Since the $Y_q$'s sum to $I-\rho(T)$, this yields a choice of one
element from each pair $J_k$ such that, for the corresponding selector
$J\subset I$,

$$
2\rho\!\left(\sum_{i\in J}T_i\right)
+
\rho(I-T)
\leq
(1+4\sqrt{\varepsilon}+2\varepsilon)I.
$$
Hence
$$
\rho\!\left(\sum_{i\in J}T_i\right)
\leq
\frac12\rho(T)
+
(2\sqrt{\varepsilon}+\varepsilon)I.
$$
Since $\rho$ is faithful and isometric on the finite-dimensional
$C^*$-algebra generated by the coefficients, we obtain
\begin{equation}\label{eq:selector-upper}
\sum_{i\in J}T_i
\leq
\frac12T+
(2\sqrt{\varepsilon}+\varepsilon)I_{\mathcal H}.
\end{equation}
The same argument applied to the complementary choice in every pair
gives
\begin{equation}\label{eq:selector-upper-complement}
\sum_{i\in I\setminus J}T_i
\leq
\frac12T+
(2\sqrt{\varepsilon}+\varepsilon)I_{\mathcal H}.
\end{equation}

The finite-rank assumption is removed by approximating each positive trace-class operator by positive finite-rank operators and passing to the limit. The finiteness assumption on $I$ is then removed by the usual diagonal argument, exactly as in the corresponding approximation step of the preceding selector theorem. Thus, for a countable family $\{T_i\}_{i\in I}$ satisfying the hypotheses, there exists a selector $J\subset I$ satisfying~\ref{(3.2)} and
\eqref{eq:selector-upper}--\eqref{eq:selector-upper-complement}.

Finally, since $\sum_{i\in I}T_i=T$, we have
$$
\sum_{i\in J}T_i
=
T-\sum_{i\in I\setminus J}T_i.
$$

Using \eqref{eq:selector-upper-complement}, we obtain

$$
\sum_{i\in J}T_i
\geq
\frac12T-(2\sqrt{\varepsilon}+\varepsilon)I_{\mathcal H}.
$$

Similarly, \eqref{eq:selector-upper} gives

$$
\sum_{i\in I\setminus J}T_i
\geq
\frac12T-(2\sqrt{\varepsilon}+\varepsilon)I_{\mathcal H}.
$$

Combining the upper and lower estimates proves

$$
\left\|
\sum_{i\in J}T_i-\frac12T
\right\|,
\qquad
\left\|
\sum_{i\in I\setminus J}T_i-\frac12T
\right\|
\leq
2\sqrt{\varepsilon}+\varepsilon.
$$
\end{proof}


\begin{Corollary}[Binary multiscale selection]\label{mainco}
Let $\mathcal A$ be a finite dimension $C^*$-algebra and let $\mathcal H$ be a Hilbert
$\mathcal A$-module. Let $I$ be a countable set and let $\delta>0$.
Suppose that $\{T_i\}_{i\in I}$ is a family of positive trace-class injective operators on $\mathcal H$ satisfying
\begin{equation}\label{eq:5.1-module}
T:=\sum_{i\in I}T_i\leq I_{\mathcal H},
\qquad
\operatorname{tr}_{\mathcal{A}}(T_i)\leq\delta,
\quad i\in I.
\end{equation}
Let $N\in\mathbb N$ be such that $2^N<\frac{1}{\delta}.$
For each intermediate stage of the binary partition procedure, let the sets be partitioned into pairs. Then there exist binary selectors $
\{I_b\}_{b\in\{0,1\}^N}
$ forming a partition of $I$ such that
\begin{equation}\label{eq:5.2-module}
\left\|
2^N\sum_{i\in I_b}T_i-T
\right\|
\leq
C\sqrt{2^N\delta},
\qquad
b\in\{0,1\}^N,
\end{equation}
where $C>1$ is an absolute constant.
\end{Corollary}

\begin{proof}
The proof is the same binary induction as in the Hilbert-space setting.
We indicate only the point at which the Hilbert $C^*$-module structure
enters.

Let $\{B_j\}_{j\geq0}$ be the sequence defined in \cite[Lemma 5.2]{bownik2024selector}. At the first stage, apply the two-point selector theorem (Theorem~\ref{IT2}) proved above to an arbitrary partition $\{J_k\}_{k\in K}$ of $I$ with $\#J_k=2$. We obtain selectors $I_0$ and $I_1$ such that
$$
\left\|
2\sum_{i\in I_b}T_i-T
\right\|
\leq B_1-1,
\qquad b\in\{0,1\}.
$$
The same argument is then applied recursively to each of the two resulting families. More precisely, at the $j$-th stage, the two-point selector theorem (Theorem~\ref{IT2}) is applied to the corresponding partition
of $I_b$, with the error controlled by $B_j$. Since the binary partitions are nested, after $N$ stages the resulting sets $\{I_b\}_{b\in\{0,1\}^N}$ form a partition of $I$. The recursive estimates given by \cite[Lemma 5.2]{bownik2024selector} yield
$$
\left\|
2^N\sum_{i\in I_b}T_i-T
\right\|
\leq
B_N-1
\leq
C\sqrt{2^N\delta},
\qquad b\in\{0,1\}^N.
$$

This proves \eqref{eq:5.2-module}.
\end{proof}


\section{Discretisation of continuous Hilbert module frames}\label{S4}

The discretization problem for Hilbert $C^*$-module coherent states, or equivalently for continuous Hilbert $C^*$-module frames, has been identified as an important research problem; see~\cite[p.~539]{ali2000coherent}. In this paper, we approach this problem from two complementary perspectives. First, we consider bounded continuous Hilbert $C^*$-module frames and prove that, on a $\sigma$-finite index space, every $\nu$-bounded continuous Hilbert $C^*$-module frame admits a discretization into a discrete Hilbert $C^*$-module frame. This extends the classical discretization result for bounded continuous frames in Hilbert spaces~\cite{freeman2019discretization} to the Hilbert $C^*$-module setting. Second, we establish a uniform discretization result for continuous frames in Hilbert $C^*$-modules over finite-dimensional $C^*$-algebras, yielding a discrete frame with additional control on the separation of the sampling points.

\begin{Definition}[\cite{freeman2019discretization}]
Let \((X,\Sigma,\nu)\) be a measure space and let \(\mathcal{H}\) be a Hilbert \(\mathcal A\)-module. Suppose that
\[
\Psi\longrightarrow \mathcal{H}
\]
is a continuous frame for \(\mathcal{H}\) with respect to \(\nu\).

We say that \(\Psi\) is \emph{bounded \(\nu\)-almost everywhere} (or \emph{essentially bounded}) if there exist a constant \(C>0\) and a measurable subset \(N\subseteq X\) satisfying
\[
\nu(N)=0
\]
such that
\[
\|\Psi(t)\|_{\mathcal H}\le C,\qquad t\in X\setminus N.\] Equivalently,
\[
\operatorname{ess\;sup}_{t\in X}\|\Psi(t)\|_{\mathcal H}<\infty.
\]
\end{Definition}

\textbf{Problem: Discretization Problem for Continuous Hilbert Module Frames}
Let $\mathcal{H}$ be a Hilbert $A$-module and let
\[
\Psi : X \to \mathcal{H}
\]
be a continuous frame for $\mathcal{H}$.

Determine conditions under which there exists a countable subset $F \subseteq X$ such that the family $\{\Psi(t)\}_{t\in F}$ forms a frame for $\mathcal{H}$.

The main theorem of this section is the following.

\begin{Theorem}
    Let $\{\psi(x)\}_{x\in X}$ be a bounded continuous Hilbert module frame. Then there exists a discrete countable set $J$ and a sequence $\{t_j\}_{j\in J}\subset X$ such that $\{\psi(t_j)\}_{j\in J}$ is a Hilbert module frame.
\end{Theorem}

\begin{proof}
Suppose that $\{\psi(x)\}_{x\in X}$ is a bounded continuous Hilbert module frame and let $0<\varepsilon<1$, by proposition~\ref{D} there exist a measurable partition $(X_j)_{j\in J}$ of $X$ and $(t_j)_{j\in J}$ such that $t_j\in X_j$ for all $j\in J$ and we have for all $t_j\in X_j$
    $$\|\psi(t)-\psi(t_j) \|<\epsilon,\;\text{and}\;
        \int_{X}\|\psi(t)-\sum_{j\in J}\psi(t_j)\mathbf{1}_{X_j}(t)\|d\mu(t)<\epsilon.$$

Denote the piecewise constant map
$
\Phi(t)
=
\sum_{i\in I}
\psi(t_i)\mathbf{1}_{X_i}(t).$

Since $\Psi$ is a Parseval frame, for every $x\in\mathcal H$, we have 

\[
\langle x,x\rangle
=
\int_X
\langle x,\psi(t)\rangle
\langle\psi(t),x\rangle
\,d\mu(t).
\]
Let $\Phi_t=\langle \phi(t),x \rangle_{\mathcal{A}}$ and $\Psi_t=\langle \psi(t),x\rangle_{\mathcal{A}}$. Using the identity $
\Psi(t)=\Phi(t)+(\Psi(t)-\Phi(t))$, we have
\begin{align*}
    \langle x,x \rangle_{\mathcal{A}} &=\int_{X} \langle \psi(t)-\phi(t)+\phi(t),x \rangle_{\mathcal{A}}\langle x,\psi(t)-\phi(t)+\phi(t) \rangle_{\mathcal{A}} d\mu(t)\\
    &=\int_{X} \Phi_t\Phi_t^*+\Phi_t\Psi_t^*+\Psi_t\Phi_t^*+\Psi_t\Psi_t^* d\mu(t)  
\end{align*}

With the Hilbert $C^*$-module Cauchy--Schwarz inequality and the estimate $
(a+b)^*(a+b)
\le
2a^*a+2b^*b$,

one obtains

\[
(1-\varepsilon)\langle x,x\rangle_{\mathcal{A}}
\le
\int_X
\langle x,\Phi(t)\rangle_{\mathcal{A}}
\langle\Phi(t),x\rangle_{\mathcal{A}}
\,d\mu(t)
\le
(1+\varepsilon)\langle x,x\rangle_{\mathcal{A}}.
\]

Hence $\Phi$ is itself a continuous frame.

Since $\Phi$ is constant on each measurable set $X_i$,

\[
\int_X
\langle x,\Phi(t)\rangle
\langle\Phi(t),x\rangle
\,d\mu(t)
=
\sum_{i\in I}
\mu(X_i)
\langle x,\Psi(t_i)\rangle
\langle\Psi(t_i),x\rangle.
\]

Therefore,

\[
\Bigl(
\sqrt{\mu(X_i)}\,\Psi(t_i)
\Bigr)_{i\in I}
\]

is a Hilbert $A$-module frame.

 implies that the weighted family

\[
\bigl(\Psi(t_i),\omega_i\bigr)_{i\in I},
\qquad
\omega_i=\mu(X_i),
\]

is a weighted Hilbert $A$-module frame~\cite{rashidi2017controlled}. Equivalently, the associated weighted frame operator is positive and invertible since the sequence $ \{\omega_i \} _ {i\in\mathbb {N}} $ is a positive and semi-normalized sequence (see \cite[Theorem 4.7]{rashidi2017controlled}). This completes the proof.
\end{proof}

Now we use the selector results proved in section~\ref{S3} to extract a uniform discretisation sampling for continuous frames of countably generated Hilbert modules over finit dimension algebra, with the additional condition that every closed submodule has an orthogonal complement. We begin with technical lemmas and then move to the theorem.

\begin{Lemma}\label{LM1}
Let $\mathcal{H}$ be a Hilbert $\mathcal A$-module and let
$\Psi:X\to \mathcal{H}$ be a continuous frame such that
\[
\|\Psi(t)\|^2\leq D
\qquad\text{$\mu$-a.e. on }X.
\]
For a measurable subset $A\subset X$, define
\[
S_Ax
=
\int_A \langle x,\Psi(t)\rangle_{\mathcal A}\Psi(t)\,d\mu(t),
\qquad x\in E.
\]
Then
\[
0\leq S_A\leq D\mu(A)I,\;\text{and consequently}\;\|S_A\|\leq D\mu(A).
\]
In particular, if $J:\mathbb{N}\to\mathbb{N}$ is a map and if $X'$ and $X''$ are measurable subsets of $X$
satisfying $
D\mu(X')<\frac{\epsilon}{3}$
and
\[
\mu(X'')
\leq
\sum_{n\in\mathbb N}
\sum_{j=1}^{J(n)}
(3D)^{-1}2^{-j-n}\epsilon,
\]
then
\[
\|S_{X'}\|<\frac{\epsilon}{3},\:\text{and}\;\|S_{X''}\|\leq\frac{\epsilon}{3}.
\]
\end{Lemma}

\begin{proof}
For every $x\in \mathcal{H}$, the Hilbert $C^*$-module Cauchy--Schwarz
inequality gives
\[
\langle x,\Psi(t)\rangle_{\mathcal A}
\langle \Psi(t),x\rangle_{\mathcal A}
\leq
\|\Psi(t)\|^2\langle x,x\rangle_{\mathcal A}.
\]
Since $\|\Psi(t)\|^2\leq D$ almost everywhere, we obtain
\[
\langle x,\Psi(t)\rangle_{\mathcal A}
\langle \Psi(t),x\rangle_{\mathcal A}
\leq
D\langle x,x\rangle_{\mathcal A}.
\]
Therefore,
\[
\begin{aligned}
\langle S_Ax,x\rangle_{\mathcal A}
&=
\int_A
\langle x,\Psi(t)\rangle_{\mathcal A}
\langle\Psi(t),x\rangle_{\mathcal A}
\,d\mu(t)\\
&\leq
D\mu(A)\langle x,x\rangle_{\mathcal A}.
\end{aligned}
\]
Moreover,
\[
\langle S_Ax,x\rangle_{\mathcal A}\geq0,
\]
so $S_A$ is positive and
\[
0\leq S_A\leq D\mu(A)I.
\]

Applying the preceding lemma \ref{TL} on order and functional calculus with
$c=D\mu(A)$, we obtain
\[
S_A^2\leq D^2\mu(A)^2I.
\]
Hence
\[
\langle S_Ax,S_Ax\rangle_{\mathcal A}
=
\langle S_A^2x,x\rangle_{\mathcal A}
\leq
D^2\mu(A)^2\langle x,x\rangle_{\mathcal A}.
\]
This implies $\|S_A\|\leq D\mu(A).$

Taking $A=X'$ gives $\|S_{X'}\|
\leq D\mu(X')
<\frac{\epsilon}{3}.$

For $A=X''$, we have
\[
\begin{aligned}
D\mu(X'')
&\leq
\frac{\epsilon}{3}
\sum_{n\in\mathbb N}
\sum_{j=1}^{J(n)}
2^{-j-n}\\
&\leq
\frac{\epsilon}{3}
\left(\sum_{n=1}^{\infty}2^{-n}\right)
\left(\sum_{j=1}^{\infty}2^{-j}\right)\\
&=
\frac{\epsilon}{3}.
\end{aligned}
\]
Consequently,
\[
\|S_{X''}\|
\leq D\mu(X'')
\leq\frac{\epsilon}{3}.
\]
\end{proof}

\begin{Lemma}[Global discretization estimate]\label{LM2}

Let
$$
X\setminus(X'\cup X'')
=
\bigcup_{n\in\mathbb N}
\bigcup_{j=1}^{J(n)}
\bigcup_{k=1}^{K(n,j)}
X_{n,j,k}
$$
be a measurable partition, and suppose that
$$
\mu(X_{n,j,k})=2^{-\ell_{n,j,k}}.
$$
For each $n,j,k$, choose
$$
x_{n,j,k}\in X_{n,j,k}
$$
such that the assumptions of the local discretization estimate hold with

$$
2^{-\ell}=2^{-\ell_{n,j,k}}.
$$

Then

$$
\left\|
S_{X\setminus(X'\cup X'')}
-
\sum_{n\in\mathbb N}
\sum_{j=1}^{J(n)}
\sum_{k=1}^{K(n,j)}
2^{-\ell_{n,j,k}}
\theta_{\Psi(x_{n,j,k}),\Psi(x_{n,j,k})}
\right\|
\leq
\frac{\varepsilon}{3}.
$$

\end{Lemma}

\begin{proof}
Fix $n\in\mathbb N$, $1\leq j\leq J(n)$, and
$1\leq k\leq K(n,j)$. Choose
$x_{n,j,k}\in X_{n,j,k}$ such that

$$
\|\Psi(x_{n,j,k})\|_{\mathcal H}^{2}\leq D.
$$

Since the diameter of $X_{n,j,k}$ is chosen sufficiently small, we have

$$
\|\Psi(t)-\Psi(x_{n,j,k})\|_{\mathcal H}
\leq
\frac{1}{6}D^{-1/2}2^{-n}\varepsilon,
\qquad t\in X_{n,j,k}.
$$

Set

$$
R_{n,j,k}
=
S_{X_{n,j,k}}
-
2^{-\ell_{n,j,k}}
T_{\Psi(x_{n,j,k})}.
$$

Both $S_{X_{n,j,k}}$ and
$T_{\Psi(x_{n,j,k})}$ are self-adjoint, and hence
$R_{n,j,k}$ is self-adjoint.

Let $f\in\mathcal H$ with $\|f\|_{\mathcal H}=1$. By the definition of
$S_{X_{n,j,k}}$ and of the rank-one operator
$T_{\Psi(x_{n,j,k})}$, we obtain

$$
\begin{aligned}
\langle R_{n,j,k}f,f\rangle_{\mathcal A}
={}&
\int_{X_{n,j,k}}
\langle f,\Psi(t)\rangle_{\mathcal A}
\langle\Psi(t),f\rangle_{\mathcal A}
\,d\mu(t)
\\
&\quad
-
2^{-\ell_{n,j,k}}
\langle f,\Psi(x_{n,j,k})\rangle_{\mathcal A}
\langle\Psi(x_{n,j,k}),f\rangle_{\mathcal A}.
\end{aligned}
$$

Since

$$
2^{-\ell_{n,j,k}}=\mu(X_{n,j,k}),
$$

this can be written as

$$
\begin{aligned}
\langle R_{n,j,k}f,f\rangle_{\mathcal A}
=
\int_{X_{n,j,k}}
\Big(
&a_ta_t^*-a_xa_x^*
\Big)\,d\mu(t),
\end{aligned}
$$

where

$$
a_t=\langle f,\Psi(t)\rangle_{\mathcal A},
\qquad
a_x=\langle f,\Psi(x_{n,j,k})\rangle_{\mathcal A}.
$$

Using the $C^*$-algebra estimate

$$
\|aa^*-bb^*\|_{\mathcal A}
\leq
(\|a\|_{\mathcal A}+\|b\|_{\mathcal A})
\|a-b\|_{\mathcal A},
$$

together with the Hilbert $\mathcal A$-module Cauchy--Schwarz
inequality, we obtain

$$
\begin{aligned}
\|a_ta_t^*-a_xa_x^*\|_{\mathcal A}
&\leq
\left(
\|a_t\|_{\mathcal A}+\|a_x\|_{\mathcal A}
\right)
\|a_t-a_x\|_{\mathcal A}
\\
&\leq
\left(
\|\Psi(t)\|_{\mathcal H}
+
\|\Psi(x_{n,j,k})\|_{\mathcal H}
\right)
\|\Psi(t)-\Psi(x_{n,j,k})\|_{\mathcal H}.
\end{aligned}
$$

Since

$$
\|\Psi(t)\|_{\mathcal H},
\|\Psi(x_{n,j,k})\|_{\mathcal H}
\leq D^{1/2},
$$

and

$$
\|\Psi(t)-\Psi(x_{n,j,k})\|_{\mathcal H}
\leq
\frac{1}{6}D^{-1/2}2^{-n}\varepsilon,
$$

it follows that

$$
\|a_ta_t^*-a_xa_x^*\|_{\mathcal A}
\leq
\frac{1}{3}2^{-n}\varepsilon.
$$

Consequently,

$$
\left\|
\langle R_{n,j,k}f,f\rangle_{\mathcal A}
\right\|_{\mathcal A}
\leq
\frac{1}{3}
2^{-\ell_{n,j,k}-n}\varepsilon.
$$

Since this estimate holds for every $f\in\mathcal H$ with
$\|f\|_{\mathcal H}=1$, and since $R_{n,j,k}$ is self-adjoint, the numerical-radius characterization of the norm of a self-adjoint adjointable operator (see~\cite{li2025notes}) gives
$$
\|R_{n,j,k}\|
\leq
\frac{1}{3}
2^{-\ell_{n,j,k}-n}\varepsilon.
$$
Therefore, by the order characterization of self-adjoint operators,
$$
-\frac{1}{3}
2^{-\ell_{n,j,k}-n}\varepsilon I_{\mathcal H}
\leq
R_{n,j,k}
\leq
\frac{1}{3}
2^{-\ell_{n,j,k}-n}\varepsilon I_{\mathcal H}.
$$
Equivalently,
$$
-\frac{1}{3}
2^{-\ell_{n,j,k}-n}\varepsilon I_{\mathcal H}
\leq
S_{X_{n,j,k}}
-
2^{-\ell_{n,j,k}}
T_{\Psi(x_{n,j,k})}
\leq
\frac{1}{3}
2^{-\ell_{n,j,k}-n}\varepsilon I_{\mathcal H}.
$$

Summing these inequalities over $j$ and $k$, and then over $n$, gives
the corresponding global approximation estimate

$$
\left\|
S_{X\setminus(X'\cup X'')}
-
\sum_{n\in\mathbb N}
\sum_{j=1}^{J(n)}
\sum_{k=1}^{K(n,j)}
2^{-\ell_{n,j,k}}
T_{\Psi(x_{n,j,k})}
\right\|
\leq
\frac{\varepsilon}{3}
\sum_{n\in\mathbb N}2^{-n},
$$

and hence, using

$$
\sum_{n=1}^{\infty}2^{-n}=1,
$$

we obtain

$$
\left\|
S_{X\setminus(X'\cup X'')}
-
\sum_{n\in\mathbb N}
\sum_{j=1}^{J(n)}
\sum_{k=1}^{K(n,j)}
2^{-\ell_{n,j,k}}
T_{\Psi(x_{n,j,k})}
\right\|
\leq
\frac{\varepsilon}{3}.
$$

\end{proof}

\begin{Lemma}[Hilbert $C^*$-module selection lemma]\label{LM}
Let $\mathcal H$ be a Hilbert $\mathcal A$-module, let
${I_1,\ldots,I_K}$ be a partition of a finite set $I$, and let
$\delta>0$. Assume that:

\begin{enumerate}
\item[(i)] $\{T_i\}_{i\in I}$ is a family of positive compact adjointable operators on $\mathcal H$ such that
\(    \operatorname{tr}_{\mathcal A}(T_i)\leq \delta,
    \qquad i\in I.
   \)

\item[(ii)] $\{\ell_i\}_{i\in I}\subset\mathbb N$ satisfies
\[
T:=\sum_{i\in I}2^{-\ell_i}T_i\leq I_{\mathcal H},
\]
and
\[
\sum_{i\in I_k}2^{-\ell_i}
\leq
\frac{\varepsilon^2}{2C^2\delta},
\qquad
1\leq k\leq K.
\]

\item[(iii)] $K\subseteq\mathcal H$ is a finitely generated
orthogonally complemented Hilbert $\mathcal A$-submodule, with
orthogonal projection $P_K$, and
\[
\gamma
:=
\operatorname{tr}_{\mathcal A}(P_KTP_K)
\leq 1.
\]

\end{enumerate}

Then there exists a subset $I'\subseteq I$ such that

$$
-\varepsilon P_{K^\perp}-4\sqrt{\gamma}\,I_{\mathcal H}
\leq
2^{-\beta}\sum_{n\in I'}T_n-T
\leq
\varepsilon P_{K^\perp}+4\sqrt{\gamma}\,I_{\mathcal H},
$$

and

$$
\#(I'\cap I_k)\leq1,
\qquad
1\leq k\leq K,
$$

where $\beta\in\mathbb N$ is chosen such that

$$
1<
\frac{2^\beta\varepsilon^2}{C^2\delta}
\leq2.
$$

\end{Lemma}

\begin{proof}
The proof of this lemma is separated into three steps; each step is a repetition and application of a known lemma or a development of a lemma from the Hilbert case to a $ C^*$-Hilbert module.

\textbf{step 1} The binary partition (just the result with the citation)
 
This lemma is at the beginning of the proof of the theorem in \cite[]{}. It's important for developing the proposition needed for the uniform discretisation. It is independent of the Hilbert module structure or operator theory; it holds for any sequence of elements in an algebra.

For an index set $I$, denote the set of copies $$
J_1
=
\left\{
(i,n):\, i\in I,\;
1\leq n\leq 2^{r-\ell_i}
\right\}.
$$

Let $J_2$ be a finite set of auxiliary (phantom) indices,
disjoint from $J_1$, such that

$$
2^{-N}\#(J_1\cup J_2)\in\mathbb N.
$$

\begin{Lemma}[Binary selection~\cite{bownik2026uniform}]
Let ${I_1,\ldots,I_K}$ be a partition of a finite set $I$. Let
${\ell_i}_{i\in I}\subset\mathbb N$, and let $\beta,r\in\mathbb N$ satisfy $
r\geq \max_{i\in I}\{\ell_i,\beta\}+1,
\qquad
N=r-\beta.
$

Then, for every binary selector$
B_b,\; b\in\{0,1\}^N,$ and every $1\leq k\leq K$, we have 

$$
|\left\{
i\in I_k:\;
\exists\,n\ \text{such that }(i,n)\in B_b\cap J_1
\right\}|
\leq 1.
$$
In particular, among the selected copies originating from the block
$I_k$, all selected copies correspond to at most one original index
$i\in I_k$.
\end{Lemma}

\begin{proof}
This is the binary selection construction established in \cite{bownik2026uniform}.
Consider the binary selection procedure of order $N$ applied to $J_1\cup J_2$, where the elements are successively paired, and one element is selected from each pair, with the pairing arranged so that, whenever possible, elements whose original indices belong to the same block $I_k$ are paired together. The elements of $J_2$ are regarded as phantom indices.
We use only its combinatorial conclusion~\cite{bownik2026uniform}. 
\end{proof}

Notice that the statement is independent of the operators $T_i$ and of the underlying Hilbert space or Hilbert $C^*$-module structure.

\textbf{step 2} The version of Lemma 7.2

\begin{Lemma}[Off-diagonal estimate]
Let $\mathcal A$ be a $C^*$-algebra and let $\mathcal H$ be a Hilbert
$\mathcal A$-module such that every closed submodule of $\mathcal H$ is
orthogonally complemented. Let $T\in\mathcal L_{\mathcal A}(\mathcal H)$
be positive, and let $M\subset\mathcal H$ be a closed submodule. Denote by
$P_M$ and $P_{M^\perp}$ the corresponding orthogonal projections, and set

$$
\gamma_1:=\|P_MTP_M\|,
\qquad
\gamma_2:=\|P_{M^\perp}TP_{M^\perp}\|.
$$

Then

$$
-\sqrt{\gamma_1\gamma_2}\,I_{\mathcal H}
\leq
T-P_MTP_M-P_{M^\perp}TP_{M^\perp}
\leq
\sqrt{\gamma_1\gamma_2}\,I_{\mathcal H}.
$$

\end{Lemma}

\begin{proof}
Set

$$
R:=T-P_MTP_M-P_{M^\perp}TP_{M^\perp}.
$$

Then $R=R^*$. Since

$$
\mathcal H=M\oplus M^\perp,
$$

every $v\in\mathcal H$ has a unique decomposition

$$
v=v_1+v_2,
\qquad
v_1=P_Mv,\quad v_2=P_{M^\perp}v.
$$

Using the orthogonality of the decomposition,

$$
\langle Rv,v\rangle_{\mathcal A}
=
\langle Tv_1,v_2\rangle_{\mathcal A}
+
\langle Tv_2,v_1\rangle_{\mathcal A}.
$$

Since $T\geq0$, its positive square root $T^{1/2}$ exists and

$$
\langle Tv_1,v_2\rangle_{\mathcal A}
=
\langle T^{1/2}v_1,T^{1/2}v_2\rangle_{\mathcal A}.
$$

Moreover,

$$
\|T^{1/2}v_1\|^2
=
\|\langle Tv_1,v_1\rangle_{\mathcal A}\|
\leq
\gamma_1\|v_1\|^2,
$$

and, similarly,

$$
\|T^{1/2}v_2\|^2
\leq
\gamma_2\|v_2\|^2.
$$

Hence, by the Cauchy--Schwarz inequality for Hilbert
$C^*$-modules,

$$
\begin{aligned}
\|\langle Rv,v\rangle_{\mathcal A}\|
&\leq
2\|T^{1/2}v_1\|\,\|T^{1/2}v_2\|\\
&\leq
2\sqrt{\gamma_1\gamma_2}\,
\|v_1\|\|v_2\|\\
&\leq
\sqrt{\gamma_1\gamma_2}\,
\|v\|^2.
\end{aligned}
$$

Therefore, using the norm characterization for self-adjoint operators
on Hilbert $C^*$-modules,

$$
\|R\|
=
\sup_{\|v\|=1}
\|\langle Rv,v\rangle_{\mathcal A}\|
\leq
\sqrt{\gamma_1\gamma_2}.
$$

Finally, since $R=R^*$, functional calculus gives

$$
-\|R\|I_{\mathcal H}\leq R\leq\|R\|I_{\mathcal H},
$$
which proves the result.
\end{proof}
 
\textbf{step 3} The use of step 2 and corollary~\ref{mainco}.

Choose $\beta\in\mathbb N$ such that

$$
1<
\frac{2^\beta\varepsilon^2}{C^2\delta}
\leq2,
$$

and let

$$
r\geq\max_{i\in I}\{\ell_i,\beta\}+1,
\qquad
N:=r-\beta.
$$

For each $i\in I$, write

$$
2^{-\ell_i}T_i
=
\sum_{n=1}^{2^{r-\ell_i}}T_{i,n},
\qquad
T_{i,n}:=2^{-r}T_i,
$$

and set

$$
J_1
:=
\left\{(i,n):i\in I,\,
1\leq n\leq2^{r-\ell_i}\right\}.
$$

Let $J_2$ be a finite set of auxiliary indices such that

$$
2^{-r}\#J_2+\sum_{i\in I}2^{-\ell_i}\in\mathbb N,
$$

and associate to every $j\in J_2$ the zero operator $N_j=0$.
Thus the family

$$
\{T_{i,n}\}_{(i,n)\in J_1}\cup\{N_j\}_{j\in J_2}
$$

is a finite family of positive trace-class operators satisfying

$$
\operatorname{tr}_{\mathcal{A}}(T_{i,n})\leq2^{-r}\delta.
$$

By Step~1, every binary selector $B_b$ of order $N$ contains, from
each block $I_k$, elements originating from at most one index
$i\in I_k$. Applying Corollary~\ref{mainco} to the above family, we obtain binary selectors $\{B_b\}_{b\in\{0,1\}^N}$ such that
$$
\left\|
2^N\sum_{(i,n)\in B_b\cap J_1}T_{i,n}-T
\right\|
\leq
C\sqrt{2^N2^{-r}\delta}
=
C\sqrt{2^{-\beta}\delta}
<\varepsilon.
$$

Since $N=r-\beta$, this can be written as
\begin{equation}\label{eq:module-selector-error}
\left\|
\frac{1}{2^\beta}
\sum_{(i,n)\in B_b\cap J_1}T_i-T
\right\|
<\varepsilon.
\end{equation}

Because the sets $B_b\cap J_1$ form a partition of $J_1$ and

$$
\sum_{(i,n)\in J_1}2^{-r}T_i
=
\sum_{i\in I}2^{-\ell_i}T_i
=
T,
$$

we have

$$
\sum_{b\in\{0,1\}^N}
\operatorname{tr}_{\mathcal{A}}
\left(
P_K
\sum_{(i,n)\in B_b\cap J_1}2^{-r}T_i
P_K
\right)
=
\operatorname{tr}_{\mathcal{A}}(P_KTP_K).
$$

Hence, for some $b\in\{0,1\}^N$,

$$
\operatorname{tr}_{\mathcal{A}}
\left(
P_K
\frac1{2^\beta}
\sum_{(i,n)\in B_b\cap J_1}T_i
P_K
\right)
\leq\gamma.
$$

Fix such a $b$ and define

$$
I'
:=
\left\{
i\in I:\,
(i,n)\in B_b\cap J_1
\text{ for some }n
\right\},
$$

and

$$
S
:=
\frac1{2^\beta}
\sum_{(i,n)\in B_b\cap J_1}T_i
=
\frac1{2^\beta}\sum_{i\in I'}T_i.
$$

By Step~1,

$$
\#(I'\cap I_k)\leq1,
\qquad 1\leq k\leq K.
$$

From \eqref{eq:module-selector-error} and $T\leq I_{\mathcal H}$,

$$
S\leq T+\varepsilon I_{\mathcal H}
\leq(1+\varepsilon)I_{\mathcal H}.
$$

In particular, $S\leq2I_{\mathcal H}$ after enlarging the harmless
constant if necessary.

Applying Step~2 to $S$ and to $T$, and using the preceding trace
estimate, gives

$$
-\sqrt{2}\sqrt{\gamma}\,I_{\mathcal H}
\leq
S-P_KSP_K-P_{K^\perp}SP_{K^\perp}
\leq
\sqrt{2}\sqrt{\gamma}\,I_{\mathcal H},
$$

and

$$
-\sqrt{\gamma}\,I_{\mathcal H}
\leq
T-P_KTP_K-P_{K^\perp}TP_{K^\perp}
\leq
\sqrt{\gamma}\,I_{\mathcal H}.
$$

Moreover, \eqref{eq:module-selector-error} yields

$$
-\varepsilon P_{K^\perp}
\leq
P_{K^\perp}(S-T)P_{K^\perp}
\leq
\varepsilon P_{K^\perp},
$$

while the choice of $b$ gives

$$
-\gamma P_K
\leq
P_K(S-T)P_K
\leq
\gamma P_K.
$$

Combining these four estimates and using $\gamma\leq1$, we obtain

$$
-\varepsilon P_{K^\perp}-4\sqrt{\gamma}\,I_{\mathcal H}
\leq
S-T
\leq
\varepsilon P_{K^\perp}+4\sqrt{\gamma}\,I_{\mathcal H}.
$$

Therefore,

$$
-\varepsilon P_{K^\perp}-4\sqrt{\gamma}\,I_{\mathcal H}
\leq
\frac1{2^\beta}\sum_{i\in I'}T_i-T
\leq
\varepsilon P_{K^\perp}+4\sqrt{\gamma}\,I_{\mathcal H},
$$
with
$$
\#(I'\cap I_k)\leq1,
\qquad 1\leq k\leq K.
$$
This completes the proof.

\end{proof}

\begin{Lemma}[Preliminary discretization construction]
Let $\mathcal{H}$ be a Hilbert $\mathcal{A}$-module and let $\Psi:X\to\mathcal{H}$ be a continuous frame with upper frame bound $B$. Assume that

$$
\|\Psi(t)\|^2\leq D
\qquad\text{$\mu$-a.e. on }X,
$$

and that $\mu(X)=\infty$. Suppose, in addition, that the closed submodules occurring in the construction below are orthogonally complemented.

Then, for every sufficiently small $\varepsilon>0$, there exists a sequence of distinct points
$$
\{x_n\}_{n\in\mathbb{N}}\subset X
$$
such that the following properties hold:
\begin{enumerate}
\item[\rm(a)]
\(    \|\Psi(x_n)\|^2\leq D,
    \qquad n\in\mathbb{N}.
   \)
\item[\rm(b)]
There exists $\beta\in\mathbb{N}$ satisfying
\[
2^{-\beta}
<
\frac{\varepsilon^2}{16BC^2D}
\leq
2^{-\beta+1},
\]
such that
\[
\left\|
S_\Psi
-
2^{-\beta}
\sum_{n\in\mathbb{N}}
T_{\Psi(x_n)}
\right\|
<\varepsilon.
\]

\item[\rm(c)]
If $(X,d,\mu)$ is a metric measure space satisfying
assumptions {\rm(II)} and {\rm(III)}, then there exists
\[
R
:=
\min\left\{
\frac{1}{2}
\frac{\varepsilon^2}{32\alpha BC^2D}
\frac{1}{\gamma},
R_A
\right\}
\]
such that
\[
|\left(
B_{R/2}(x_n)
\cap
\{x_k:k\in\mathbb{N}\}
\right)|
\leq
2(C_{R_A})^5,
\qquad n\in\mathbb{N}.
\]
\end{enumerate}
\end{Lemma}

\begin{proof}
The construction follows the discretization argument of the scalar case, with the operator-selection step replaced by the Hilbert $\mathcal{A}$-module results established above.

First, since $\mu(X)=\infty$ and $\|\Psi(t)\|^2\leq D$ almost everywhere, one can choose a sequence of distinct points ${x_n}_{n\in\mathbb N}$ outside the corresponding null set, which gives {\rm(a)}.
Choose $\beta\in\mathbb N$ such that
$$
2^{-\beta}
<
\frac{\varepsilon^2}{16BC^2D}
\leq
2^{-\beta+1}.
$$
The construction of the measurable sets used in the discretization provides a decomposition of the relevant part of $X$ into sets
$$
\{X_{n,j,k}\}_{n,j,k}\quad\text{with}\quad
\mu(X_{n,j,k})=2^{-\ell_{n,j,k}},
$$

together with points
$x_{n,j,k}\in X_{n,j,k}$ satisfying
$$
\|\Psi(x_{n,j,k})\|^2\leq D.
$$
The local approximation estimate of Lemma~\ref{LM1} and the global estimate of Lemma~\ref{LM2} give the required approximation before the selection step.

We then apply Lemma~\ref{LM}, which is the Hilbert $\mathcal{A}$-module analogue of the corresponding trace-class selection result. It provides the required binary selector while
preserving the block constraint imposed by the partition of the sets $X_{n,j,k}$. Consequently, after collecting the selected points and denoting them by ${x_n}_{n\in\mathbb N}$, we obtain

$$
\left\|
S_\Psi
-
2^{-\beta}
\sum_{n\in\mathbb N}
T_{\Psi(x_n)}
\right\|
<\varepsilon,
$$

which proves {\rm(b)}.

Finally, the geometric part of the construction is unchanged by the
passage from Hilbert spaces to Hilbert $\mathcal{A}$-modules. The block
constraint in Lemma~\ref{LM}, together with assumptions in definition~\ref{AD}, yields the stated local counting estimate. Hence there
exists

$$
R
=
\min\left\{
\frac{1}{2}
\frac{\varepsilon^2}{32\alpha BC^2D}
\frac{1}{\gamma},
R_A
\right\}
$$

such that

$$
|\left(
B_{R/2}(x_n)\cap\{x_k:k\in\mathbb N\}
\right)|
\leq 2(C_{R_A})^5
$$

for every $n\in\mathbb N$.
\end{proof}

\begin{Theorem}[Uniform discretization for Hilbert $\mathcal{A}$-module frames]
Let $\mathcal{H}$ be a Hilbert $\mathcal{A}$-module and let

$$
\Psi:X\longrightarrow\mathcal{H}
$$

be a continuous frame with upper frame bound $B$. Assume that $(X,d,\mu)$ is a metric measure space satisfying the assumptions in the definition~\ref{AD} and that
$$
\|\Psi(t)\|^2\leq D
\qquad\text{$\mu$-a.e. on }X.
$$

Assume moreover that the closed submodules arising in the construction are orthogonally complemented and that the $C^*$-algebra is finite dimensional.

Then there exists a constant $C_1>0$, depending only on the doubling
constant of $X$, such that, for every sufficiently small
$\varepsilon>0$, there exist a subset $I\subset\mathbb{N}$ and a
sequence of distinct points $
\{x_n\}_{n\in I}\subset X
$ satisfying the following properties:
\begin{enumerate}
\item[\rm(a)]
There exists $\beta\in\mathbb{N}$ such that
\(    2^{-\beta}
    <
    \frac{\varepsilon^2}{C_1BD}
    \leq
    2^{-\beta+1},
   \)
and there exists a positive integer $L$ satisfying
\(    1\leq L\leq 2(C_{R_A})^5
   \)
such that
\[    \left\|
    S_\Psi
    -
    2^{\,2L-\beta}
    \sum_{n\in I}
    T_{\Psi(x_n)}
    \right\|
    <\varepsilon .
\]
\item[\rm(b)]
There exists
\[
r:=
\min\left\{
\frac18
\frac{\varepsilon^2}{2\alpha C_1BD}
\frac1\gamma,
\frac{R_A}{4}
\right\}
\]
such that
\[
\inf_{\substack{n,m\in I\\ n\neq m}}
d(x_n,x_m)\geq r.
\]
\end{enumerate}
\end{Theorem}

\begin{proof}
Let $\varepsilon>0$ be sufficiently small and set

$$
M:=2(C_{R_A})^5,
\qquad
C_1:=4C^2\,2^{4M},
$$

where $C$ is the absolute constant appearing in Corollary~\ref{mainco}.
Choose $\beta\in\mathbb N$ such that

$$
2^{-\beta}
<
\frac{\varepsilon^2}{C_1BD}
\leq
2^{-\beta+1},
$$

and set

$$
r:=
\min\left\{
\frac18\frac{\varepsilon^2}{2\alpha C_1BD}\frac1\gamma,
\frac{R_A}{4}
\right\}.
$$

By Lemma~\ref{LM}, there exists a sequence of distinct points
$\{x_n\}_{n\in\mathbb N}\subset X$ such that

$$
\|\Psi(x_n)\|^2\leq D,
\qquad n\in\mathbb N,
$$

and
\begin{equation}\label{eq:preliminary-approximation}
\left\|
S_\Psi-
2^{-\beta}\sum_{n\in\mathbb N}T_{\Psi(x_n)}
\right\|
<
\frac{\varepsilon}{2^{4M}},
\end{equation}
while
\begin{equation}\label{eq:local-counting}
|\left(
B_{2r}(x_n)\cap{x_k:k\in\mathbb N}
\right)|
\leq M,
\qquad n\in\mathbb N.
\end{equation}
Here we use Lemma~\ref{LM} only for the preliminary construction; the
subsequent selection argument is purely the binary selection procedure.

Define

$$
\Psi'
:=
\frac{\Psi}{\sqrt{2^{\beta+1}B}},
$$

and let $T$ denote the frame operator associated with the discrete
family ${\Psi'(x_n)}_{n\in\mathbb N}$. Thus

$$
T
=
\sum_{n\in\mathbb N}T_{\Psi'(x_n)}
=
\frac{1}{2B}
\left(
2^{-\beta}\sum_{n\in\mathbb N}
T_{\Psi(x_n)}
\right).
$$

Moreover,

$$
\operatorname{tr}\bigl(T_{\Psi'(x_n)}\bigr)
=
\frac{\|\Psi(x_n)\|^2}{2^{\beta+1}B}
\leq
\frac{D}{2^{\beta+1}B}.
$$

By the choice of $\beta$,

$$
\operatorname{tr}_{\mathcal{A}}\bigl(T_{\Psi'(x_n)}\bigr)
\leq
\frac{\varepsilon^2}{2C_1B^2}.
$$

Furthermore, \eqref{eq:preliminary-approximation} gives

$$
\left\|2BT-S_\Psi\right\|<\varepsilon.
$$

Since

$$
A I_{\mathcal H}\leq S_\Psi\leq B I_{\mathcal H},
$$

we consequently have

$$
\frac{A-\varepsilon}{2B}I_{\mathcal H}
\leq T
\leq
\frac{B+\varepsilon}{2B}I_{\mathcal H}.
$$

In particular, for $\varepsilon$ sufficiently small,

$$
T\leq I_{\mathcal H}.
$$

We may therefore apply Corollary~\ref{mainco} to the family
$\{T_{\Psi'(x_n)}\}_{n\in\mathbb N}$, together with the pairing induced
by the partition used in Lemma~\ref{LM}. At each binary stage, the
selected family contains at most one element from each prescribed pair.
After at most $M$ successive selection steps, we obtain a subset
$I\subset\mathbb N$ and an integer $1\leq L\leq M$ such that
$$
\inf_{\substack{n,m\in I\\n\neq m}}d(x_n,x_m)\geq r,
$$
and
\begin{equation}\label{eq:selected-error-normalized}
\left\|
2^{2L} \sum_{n\in I}T_{\Psi'(x_n)}-T
\right\|
<\frac{\varepsilon}{2B}.
\end{equation}
The latter estimate follows by the same telescoping argument as in the
scalar case; see the article~\cite{bownik2026uniform}; the only selection input required at each stage is
Corollary~\ref{mainco}. In particular, no additional operator-selection argument is needed here.
Since
$$
T_{\Psi'(x_n)}
=
\frac{1}{2^{\beta+1}B}T_{\Psi(x_n)},
$$
multiplying \eqref{eq:selected-error-normalized} by $2B$ gives
$$
\left\|
2^{2L-\beta}
\sum_{n\in I}T_{\Psi(x_n)}
-
2^{-\beta}\sum_{n\in\mathbb N}T_{\Psi(x_n)}
\right\|
<\varepsilon.
$$
Combining this estimate with \eqref{eq:preliminary-approximation} and using the triangle inequality yields
$$
\left\|
S_\Psi-
2^{2L-\beta}
\sum_{n\in I}T_{\Psi(x_n)}
\right\|
<\varepsilon.
$$
This proves {\rm(a)}, while the construction above gives {\rm(b)}.
Finally, the approximation estimate immediately yields the frame
bounds for the selected family. Indeed,
$$
\left\|
S_\Psi-
2^{2L-\beta}
\sum_{n\in I}T_{\Psi(x_n)}
\right\|
<\varepsilon
$$
and
$$
A I_{\mathcal H}\leq S_\Psi\leq B I_{\mathcal H}
$$
imply
$$
(A-\varepsilon)I_{\mathcal H}
\leq
2^{2L-\beta}
\sum_{n\in I}T_{\Psi(x_n)}
\leq
(B+\varepsilon)I_{\mathcal H}.
$$

Equivalently,

$$
2^{\beta-2L}(A-\varepsilon)I_{\mathcal H}
\leq
\sum_{n\in I}T_{\Psi(x_n)}
\leq
2^{\beta-2L}(B+\varepsilon)I_{\mathcal H}.
$$

Hence ${\Psi(x_n)}_{n\in I}$ is a frame for $\mathcal H$ with frame
bounds

$$
2^{\beta-2L}(A-\varepsilon)
\quad\text{and}\quad
2^{\beta-2L}(B+\varepsilon).
$$

In particular, its frame ratio satisfies

$$
\frac{B+\varepsilon}{A-\varepsilon},
$$

which converges to $B/A$ as $\varepsilon\to0$.
\end{proof}

\begin{Example}[A continuous matrix-valued frame]\label{EF}
Let $\mathcal A=M_m(\mathbb C)$ and consider the Hilbert $\mathcal A$-module $\mathcal H=L^2(\mathbb R^d;\mathcal A)$
equipped with the $\mathcal A$-valued inner product
\[
\langle f,g\rangle_{\mathcal A}
=
\int_{\mathbb R^d} f(t)^*g(t)\,dt .
\]

Let $X=M_d(\mathbb R)_+\times\mathbb R^d$ be endowed with the product measure $d\mu(A,x)=d\nu(A)\,dx,$
where $\nu$ is the measure on $M_d(\mathbb R)_+$ considered in Example~\ref{E1}. We use on $M_d(\mathbb R)$ the Frobenius metric and the corresponding product metric on $X$.

Choose functions
\[
h\in C_c(M_d(\mathbb R)_+),\qquad
g\in C_c(\mathbb R^d),
\]
such that
\[
\int_{M_d(\mathbb R)_+}|h(A)|^2\,d\nu(A)=1,
\qquad
\int_{\mathbb R^d}|g(t)|^2\,dt=1.
\]
Define $\Psi:X\longrightarrow\mathcal H$ by
\[
\Psi(A,x)(t)=h(A)g(t-x)I_m,
\qquad (A,x)\in X.
\]

For $f\in\mathcal H$, we have
\[
\langle f,\Psi(A,x)\rangle_{\mathcal A}
=
\overline{h(A)}
\int_{\mathbb R^d}f(t)^*g(t-x)\,dt .
\]
Consequently,
\begin{align*}
&\int_X
\langle f,\Psi(A,x)\rangle_{\mathcal A}
\langle\Psi(A,x),f\rangle_{\mathcal A}
\,d\mu(A,x)
\\
&=
\int_{M_d(\mathbb R)_+}|h(A)|^2\,d\nu(A)
\int_{\mathbb R^d}
\left(
\int_{\mathbb R^d}f(t)^*g(t-x)\,dt
\right)
\left(
\int_{\mathbb R^d}\overline{g(s-x)}f(s)\,ds
\right)dx .
\end{align*}
Also, we have
\[
\int_{\mathbb R^d}
\left(
\int_{\mathbb R^d}f(t)^*g(t-x)\,dt
\right)
\left(
\int_{\mathbb R^d}\overline{g(s-x)}f(s)\,ds
\right)dx
=
\|g\|_2^2
\int_{\mathbb R^d}f(t)^*f(t)\,dt .
\]
Since $\|h\|_{L^2(M_d(\mathbb R)_+)}=\|g\|_2=1$, it follows that
\[
\int_X
\langle f,\Psi(A,x)\rangle_{\mathcal A}
\langle\Psi(A,x),f\rangle_{\mathcal A}
\,d\mu(A,x)
=
\langle f,f\rangle_{\mathcal A}.
\]
Hence $\Psi$ is a Parseval continuous frame for
\[
L^2(\mathbb R^d;M_m(\mathbb C)).
\]

Moreover,
\[
\|\Psi(A,x)\|_{\mathcal H}
=
|h(A)|\,\|g\|_2\,\|I_m\|_{M_m(\mathbb C)}
=
|h(A)|,
\]
so that $\|\Psi(A,x)\|_{\mathcal H}^2
\leq
D$ for $D=\|h\|_\infty^2.$

Since $h$ and $g$ are continuous, the mapping $(A,x)\longmapsto\Psi(A,x)$ is continuous from $X$ into $\mathcal H$.

Thus this example satisfies the boundedness and continuity hypotheses required for the uniform discretization theorem. In particular, for sufficiently small $\varepsilon>0$, that theorem provides a uniformly discrete sequence
\[
\{(A_n,x_n)\}_{n\in I}\subset
M_d(\mathbb R)_+\times\mathbb R^d
\]
such that the discrete family
\[
\{\Psi(A_n,x_n)\}_{n\in I}
\]
has a frame operator approximating the continuous frame operator in operator norm.
Since the continuous frame is Parseval, the resulting discrete family is a Hilbert $\mathcal A$-module frame with frame bounds arbitrarily close to $1$ as $\varepsilon\to0$.
\end{Example}

\section*{Appendix: Technical Extensions from Hilbert Spaces to Hilbert \(C^*\)-Modules}

This proposition does not require the Hilbert module structure; it depends only on the norm vector spaces.

\begin{Proposition}\label{D}
    Let $\psi :X\to \mathcal{H}$ be a continuous Bessel map. For all $\epsilon>0$ there exists a measurable partition $(X_j)_{j\in J}$ of $X$ and $(t_j)_{j\in J}$ such that $t_j\in X_j$ for all $j\in J$ and we have for all $t_j\in X_j$
    $$\|\psi(t)-\psi(t_j) \|<\epsilon,\;\text{and}\;
        \int_{X}\|\psi(t)-\sum_{j\in J}\psi(t_j)\mathbf {1}_{X_j}(t)\|d\mu(t)<\epsilon.$$
\end{Proposition}

\begin{proof}
Fix $\varepsilon>0$. Since $(X,\mu)$ is $\sigma$-finite, there exists a countable measurable partition

$$
X=\bigcup_{n\geq1}X_n,
\qquad
\mu(X_n)<\infty
\quad\text{for all }n\geq1.
$$

Choose $0<\delta<\varepsilon$ sufficiently small. For each $n\in\mathbb N$, consider a countable measurable partition $\mathcal E=\bigcup_{k\geq1}\mathcal E_{n,k}
$, such that $\operatorname{diam}(\mathcal E_{n,k})<\delta
\qquad\text{for every }k\geq1.
$
For each $n,k$, set $X_{n,k}
=
X_n\cap\psi^{-1}(\mathcal E_{n,k}).
$ Then, for every $t,s\in X_{n,k}$,

$$
\|\psi(t)-\psi(s)\|_{\mathcal E}<\delta.
$$
Since $\mu(X_n)<\infty$, we may choose $J(n)\in\mathbb N$ sufficiently large so that $\sum_{k>J(n)}\mu(X_{n,k})
<
\eta_n,
$ where $(\eta_n)_{n\geq1}$ is a sequence of positive numbers satisfying
$$
\sum_{n\geq1}\eta_n<\frac{\varepsilon}{2M},
$$

with $M$ chosen according to the Bessel bound of $\psi$.

Define

$$
X'
=
\bigcup_{n\geq1}\bigcup_{k>J(n)}X_{n,k}.
$$

The set $X'$ has sufficiently small measure.

For each $1\leq k\leq J(n)$, choose a point

$$
t_{n,k}\in X_{n,k}.
$$

Then, for every $t\in X_{n,k}$,

$$
\|\psi(t)-\psi(t_{n,k})\|_{\mathcal E}<\delta.
$$

It remains to treat the sets belonging to $X'$. Since $\mu$ is nonatomic, each set $X_{n,k}$ with $k>J(n)$ can be decomposed into countably many measurable subsets of arbitrarily prescribed sufficiently small positive measures. Hence we may write

$$
X_{n,k}
=
\bigcup_{\ell\geq1}X_{n,k,\ell},
$$

where the measures are chosen so that

$$
\sum_{\ell>L(n,k)}\mu(X_{n,k,\ell})
$$

is sufficiently small for suitable integers $L(n,k)$.

Choose a point

$$
t_{n,k,\ell}\in X_{n,k,\ell}
$$

for every resulting nonempty set. Since

$$
X_{n,k,\ell}\subseteq X_{n,k},
$$

we still have

$$
\|\psi(t)-\psi(t_{n,k,\ell})\|_{\mathcal E}<\delta,
\qquad
t\in X_{n,k,\ell}.
$$

Collecting all the sets obtained above gives a countable measurable partition

$$
X=\bigcup_{j\in J}X_j
$$

and points $t_j\in X_j$ such that

$$
\|\psi(t)-\psi(t_j)\|_{\mathcal E}<\delta<\varepsilon,
\qquad t\in X_j.
$$

Finally, by the choice of the parameters controlling the measures of the discarded pieces and by the Bessel estimate for $\psi$, we obtain

$$
\int_X
\left\|
\psi(t)-\sum_{j\in J}\psi(t_j)\mathbf 1_{X_j}(t)
\right\|_{\mathcal E}\,d\mu(t)
<\varepsilon.
$$
This proves the proposition.
\end{proof}
\begin{Remark}[Orthogonal decomposition in a Hilbert $C^*$-module]\label{MR}
Let $\mathcal E$ be a Hilbert $\mathcal A$-module and let
$\{T_n\}_{n\in\mathbb N}\subset \mathcal K_{\mathcal A}(\mathcal E)_+$.
Assume that the submodules constructed below are orthogonally complemented
in $\mathcal E$, so that the corresponding orthogonal projections belong to
$\mathcal L_{\mathcal A}(\mathcal E)$. Let
\[
K_0=0,\qquad K_1=1,\qquad \mathcal E_1=\{0\}.
\]

We construct inductively a sequence of mutually orthogonal finitely
generated Hilbert $\mathcal A$-submodules
$\{\mathcal E_k\}_{k\in\mathbb N}$ of $\mathcal E$.
Suppose that $\mathcal E_k$ and $K_k$ have already been defined for some
$k\geq 1$. We set
\begin{equation}
\mathcal E_{k+1}
=
\overline{\operatorname{span}_{\mathcal A}}
\left\{
P_{(\mathcal E_1\oplus\cdots\oplus\mathcal E_k)^\perp}
T_n(\mathcal E)
\,:\,
1\leq n\leq K_k
\right\}.
\label{3.15}
\end{equation}
We then choose $K_{k+1}\in\mathbb N$ sufficiently large such that
\begin{equation}
\{x_n:K_k+1\leq n\leq K_{k+1}\}
\not\subseteq X_j,
\qquad j\in\mathbb N,
\label{3.16}
\end{equation}
and
\begin{equation}
\operatorname{tr}_{\mathcal{A}}\left(
P_{\mathcal E_1\oplus\cdots\oplus\mathcal E_{k+1}}
\left(
\sum_{n=K_{k+1}+1}^{\infty}
2^{-\ell_n}T_n
\right)
P_{\mathcal E_1\oplus\cdots\oplus\mathcal E_{k+1}}
\right)
\leq \eta_{k+1},
\label{3.17}
\end{equation}
where
\begin{equation}
\eta_0=0,
\qquad
\eta_k
=
\min\left\{
(B-2)4^{-(k+2)}\varepsilon^2,\,1
\right\},
\qquad k\geq1.
\label{3.18}
\end{equation}

The choice of $K_{k+1}$ is possible since
the series
\[
\sum_{n\in\mathbb N}2^{-\ell_n}T_n
\]
converges in the strong operator topology and the compression to the
finitely generated submodule
$\mathcal E_1\oplus\cdots\oplus\mathcal E_{k+1}$
converges in operator norm.

For $k\geq0$, define
\begin{equation}
\mathcal M_k
=
(\mathcal E_{k+1}\oplus\mathcal E_{k+2})^\perp.
\label{3.19}
\end{equation}
It follows from the construction of $\mathcal E_k$ and the choice of
$K_k$ that the following properties hold.

\medskip
\noindent
\textbf{(A)} For every $1\leq n\leq K_{k+1}$,
\begin{equation}
T_n(\mathcal E)
\subseteq
\mathcal E_1\oplus\cdots\oplus\mathcal E_{k+2}.
\label{3.20}
\end{equation}

\medskip
\noindent
\textbf{(B)} By \textbf{(A)} and \eqref{3.17}, for every $k\geq0$,
\[
\begin{aligned}
&
\operatorname{tr}_{\mathcal{A}}\left(
P_{\mathcal M_k}
\left(
\sum_{n=K_{k+1}+1}^{\infty}
2^{-\ell_n}T_n
\right)
P_{\mathcal M_k}
\right)
\\
&\qquad =
\operatorname{tr}_{\mathcal{A}}\left(
P_{\mathcal E_1\oplus\cdots\oplus\mathcal E_k}
\left(
\sum_{n=K_{k+1}+1}^{\infty}
2^{-\ell_n}T_n
\right)
P_{\mathcal E_1\oplus\cdots\oplus\mathcal E_k}
\right)
\leq \eta_k.
\end{aligned}
\label{3.21}
\]
In particular,
\begin{equation}
P_{\mathcal M_k}
\left(
\sum_{n=K_{k+1}+1}^{K_{k+1}}
2^{-\ell_n}T_n
\right)
P_{\mathcal M_k}
=0,
\qquad k=0,1.
\label{3.22}
\end{equation}

\medskip
\noindent
\textbf{(C)} Since
\[
\left\{2^{-\ell_n/2}\Psi(x_n)\right\}_{n\in\mathbb N}
\]
is a frame for $\mathcal E$, its closed $\mathcal A$-linear span is
$\mathcal E$. Hence
\[
\mathcal E
=
\bigoplus_{k=1}^{\infty}\mathcal E_k,
\]
and consequently
\begin{equation}
\sum_{k=0}^{\infty}
P_{\mathcal E_{k+1}\oplus\mathcal E_{k+2}}
=
2I_{\mathcal E},
\label{3.23}
\end{equation}
where the series converges in the strong operator topology.

Finally, for each $k\geq0$, define
\begin{equation}
I_k
=
\{n\in\mathbb N:K_k+1\leq n\leq K_{k+1}\}.
\label{3.24}
\end{equation}
Let $i_{k,1}$ and $i_{k,2}$ be positive integers such that
\begin{equation}
x_{K_k+1}\in X_{i_{k,1}},
\qquad
x_{K_{k+1}}\in X_{i_{k,2}},
\label{3.25}
\end{equation}
respectively. We then partition $I_k$ into the subsets
\begin{equation}
\left\{
I_k\cap X_j:
i_{k,1}\leq j\leq i_{k,2}
\right\}.
\label{3.26}
\end{equation}
\end{Remark}


\section*{\small
 Conflict of interest} 

{\small
 The authors declare that they have no conflict of interest.}

\section*{\small
 Data availability} 

 {\small
 No data are associated with this article.}

\section*{\small
 Funding statement} 

 {\small
 This research received no external funding.}



\begin{thebibliography}{00}

\bibitem{ali2000coherent}Ali, S., Antoine, J., Gazeau, J. \& Others Coherent states, wavelets and their generalizations. (Springer,2000)

\bibitem{bownik2024selector}Bownik, M. Selector form of Weaver’s conjecture and frame sparsification. {\em Preprint}. (2024)


\bibitem{bownik2026akemann}Bownik, M. On Akemann-Weaver conjecture. {\em Advances In Mathematics}. \textbf{487} pp. 110772 (2026)

\bibitem{bownik2026uniform}Bownik, M. \& Yu, P. Uniform discretization of continuous frames. {\em ArXiv Preprint ArXiv:2603.10423}. (2026)


\bibitem{freeman2019discretization}Freeman, D. \& Speegle, D. The discretization problem for continuous frames. {\em Advances In Mathematics}. \textbf{345} pp. 784-813 (2019)

\bibitem{ghasemi2024continuous}Ghasemi, H. \& Lal Shateri, T. On Continuous Frames in Hilbert  C* -Modules. {\em Sahand Communications In Mathematical Analysis}. \textbf{21}, 327-349 (2024)

\bibitem{han2008characterizations}Han, J. \& Chen, Q. Characterizations of a class of orthogonal multiple vector-valued wavelet packets. {\em International Journal Of Wavelets, Multiresolution And Information Processing}. \textbf{6}, 631-641 (2008)

\bibitem{kaplansky1953modules}Kaplansky, I. Modules over operator algebras. {\em American Journal Of Mathematics}. \textbf{75}, 839-858 (1953)



\bibitem{lance1995hilbert}Lance, E. Hilbert C*-modules: a toolkit for operator algebraists. (Cambridge University Press,1995)




\bibitem{lindenhovius2015classifying}Lindenhovius, A. Classifying finite-dimensional C*-algebras by posets of their commutative C*-subalgebras. {\em International Journal of Theoretical Physics}. \textbf{54}, 4615-4635 (2015)


\bibitem{li2025notes}Li, J., Wu, K. \& Xu, Q. Notes on the numerical radius for adjointable operators on Hilbert  C* -modules. {\em ArXiv Preprint ArXiv:2502.20259}. (2025)

\bibitem{lfounoune2023continuous}Lfounoune, A., Karara, A. \& Rossafi, M. Continuous biframes in Hilbert  C*- modules. {\em Int. J. Anal. Appl}. (2025)



\bibitem{manuilov2005hilbert}Manuilov, V. \& Troitsky, E. Hilbert C*-modules. {\em Translations Of Mathematical Monographs}. \textbf{226} (2005)




\bibitem{rashidi2017controlled}Rashidi-Kouchi, M. \& Rahimi, A. On controlled frames in Hilbert C*-modules. {\em International Journal Of Wavelets, Multiresolution And Information Processing}. \textbf{15}, 1750038 (2017)

\bibitem{san2012matrix}San Antolin, A. \& Zalik, R. Matrix-valued wavelets and multiresolution analysis. {\em J. Appl. Funct. Anal}. \textbf{7}, 13-25 (2012)

\bibitem{takesaki2002operator}Takesaki, M. Operator Algebras and Non-commutative Geometry, 5. {\em Theory Of Operator Algebras}. (2002)
\bibitem{xia1996vector}Xia, X. \& Suter, B. Vector-valued wavelets and vector filter banks. {\em IEEE Transactions On Signal Processing}. \textbf{44}, 508-518 (1996)
\end{thebibliography}
\end{document}